\documentclass{amsart}
\usepackage{latexsym,amsmath,amssymb,amscd}
\usepackage{eucal}

\usepackage{color}

\newcommand\rp{'}

\newcommand\dlp{``}
\newcommand\drp{''}

\newcommand\umlaut{\"}

\theoremstyle{plain}

\newtheorem{theorem}{Theorem}[section]
\newtheorem{corollary}[theorem]{Corollary}

\newtheorem{lemma}[theorem]{Lemma}
\newtheorem{proposition}[theorem]{Proposition}

\theoremstyle{definition}
\newtheorem{definition}[theorem]{Definition}

\theoremstyle{remark}
\newtheorem{remark}[theorem]{Remark}

\newtheorem{example}[theorem]{Example}

\newcommand{\larray}{\left(\begin{array}{cc}\right.}
\newcommand{\rarray}{\right.\end{array}\right)}

\DeclareMathOperator{\Hom}{Hom}

\DeclareMathOperator{\HC}{HC}
\DeclareMathOperator{\Hd}{HH}
\DeclareMathOperator{\HP}{HP}
\newcommand\conj[1]{\langle #1\rangle}

\newcommand\Spec{\operatorname{Spec}}
\newcommand\Prim{\operatorname{Prim}}

\newcommand\Tt{a_0 \otimes a_1 \otimes \ldots \otimes a_n}

\newcommand{\bdsplit}{\begin{displaymath}
\begin{split}}
\newcommand{\edsplit}{\end{split}
\end{displaymath}}

\newcommand{\beqn}{\begin{equation}}
\newcommand{\eeqn}{\end{equation}}
\newcommand{\bsplit}{\begin{split}}
\newcommand{\esplit}{\end{split}}

\newcommand{\calb}{\mathcal{B}}
\newcommand{\calc}{\mathcal{C}}

\newcommand{\calo}{\mathcal{O}}

\newcommand\CI{\mathcal{C}\sp{\infty}}
\newcommand\Tor{\operatorname{Tor}}
\newcommand\Hc{\operatorname{HC}}
\newcommand\Hp{\operatorname{HP}}
\newcommand\tHp{\operatorname{HP}^{\topo}}
\newcommand\tHc{\operatorname{HC}^{\topo}}
\newcommand\tHd{\operatorname{HH}^{\topo}}
\newcommand\topo{\operatorname{top}}
\newcommand\opp{\operatorname{op}}
\newcommand\mfk{\mathfrak}

\newcommand\pa{\partial}

\newcommand\supp{\operatorname{supp}}
\newcommand\cohom{\operatorname{H}}
\newcommand\Aut{\operatorname{Aut}}

\newcommand\CC{\mathbb C}

\newcommand\KK{\mathbb K}

\newcommand\kk{\mathbf k}

\newcommand\TT{\mathbb T}
\newcommand\ZZ{\mathbb Z}
\begin{document}
\title[Crossed product algebras]%
{The periodic cyclic homology of crossed products of
finite type algebras}
\author[J. Brodzki]{Jacek Brodzki}
\address{Deparment of Mathematical Sciences 
University of Southampton 
Southampton SO17 1BJ U.K.}
\email{j.brodzki@soton.ac.uk}

\author[S. Dave]{Shantanu Dave} \address{Wolfgang Pauli Institute,
  Vienna, 1090 Austria} \email{shantanu.dave@univie.ac.at}

\author[V. Nistor]{Victor Nistor}
\address{ D\'epartement de Math\'ematiques, Universit\'{e} de
  Lorraine, 57045 METZ, France
%
%
and Inst. Math. Romanian Acad.  PO BOX 1-764, 014700 Bucharest
Romania} 
\email{victor.nistor@univ-lorraine.fr}

\thanks{J.~B.~was partially supported by an EPSRC grant EP/I016945/1. S. D. was supported by grant P 24420-N25 of the
  Austrian Science Fund (FWF). V. N. was supported in part by ANR-14-CE25-0012-01
  (SINGSTAR)}

\begin{abstract}
We study the periodic cyclic homology groups of the cross-product of a
finite type algebra $A$ by a discrete group $\Gamma$. In case $A$ is
commutative and $\Gamma$ is finite, our results are complete and given
in terms of the singular cohomology of the sets of fixed
points. These groups identify our cyclic homology groups with the 
``orbifold cohomology'' of the underlying (algebraic) orbifold. The
proof is based on a careful study of localization at fixed points and
of the resulting Koszul complexes. We provide examples of Azumaya
algebras for which this identification is, however, no longer
valid. As an example, we discuss some affine Weyl groups.
\end{abstract}
\maketitle

\date{\today}
\tableofcontents

\section*{Introduction}

Let $A$ be an algebra and let $\Gamma$ be a (discrete) group acting on
it by a morphism $\alpha : \Gamma \to \Aut(A)$, where $\Aut(A)$
denotes the group of automorphisms of $A$. Then, to this action of
$\Gamma$ on $A$, we can associate the crossed product algebra $A
\rtimes \Gamma$ consisting of finite sums of elements of the form $a
\gamma$, $a \in A, ~\gamma\in \Gamma$, subject to the relations
\begin{equation}
  \gamma a \, = \, \alpha_{\gamma} (a) \gamma \,.
\end{equation}
Crossed products appear often in algebra and analysis and can be used
to model several geometric structures. They play a fundamental role in
non-commutative geometry \cite{ConnesBook}. For instance,
cross-products can be used to recover equivariant $K$-theory
\cite{BaumConnes2, BaumConnes1, ConnesSkandalis, 
FeiginTsygan, Julg}.

Cyclic homology is a homological theory for algebras that can be used
to recover the de-Rham cohomology of a smooth, compact manifold $M$ as
the periodic cyclic homology of the algebra $\CI(M)$ of smooth
functions on $M$ \cite{ConnesIHES} (see also \cite{ConnesBook,
  Karoubi,  LodayBook, LodayQuillen, ManinNCG, Tsygan}). In this paper,
we are interested in the algebraic counterpart of this result
\cite{FeiginTsygan} (see also \cite{BurgheleaComm, EmmanouilCR,
  Emmanouil, Farinati, Lorenz}), in view of its connections with the representation
theory of reductive $p$-adic groups (see Section \ref{sec.weyl} for
references and further comments), and we provide some more general
results as well.

Let us assume that our algebra $A$ is an algebra over a ring $\kk$. So
$1 \in \kk$, but $A$ is not required to have a unit. Nevertheless, for
simplicity, in this paper, we shall restrict to the case when $A$ has
a unit.  Let us assume that the group $\Gamma$ acts in a compatible
way on both the algebra $A$ and the ring $\kk$. Using $\alpha$ to
denote the action on $\kk$ as well, this means that
\begin{equation*}
   \alpha_\gamma(fa) \, = \, \alpha_\gamma(f) \alpha_\gamma(a)
   \,, \quad \mbox{for all }\ f \in \kk \ \mbox{ and }\ a \in A \,.
\end{equation*}

In this paper we study the Hochschild, cyclic, and periodic cyclic
homology groups of $A \rtimes \Gamma$ using the additional information
provided by the action of $\Gamma$ on $\kk$. We begin with some
general results and then particularize, first, to the case when $A$ is
a finite type algebra over $\kk$ and then, further, to the case when
$A = \kk$. (Recall \cite{KazhdanNistorSchneider} that $A$ is a finite
type algebra over $\kk$ if $A$ is a $\kk$-algebra, $\kk$ is a quotient
of a polynomial ring, and $A$ is finitely generated as a
$\kk$-module.)

For $\gamma \in \Gamma$, let us denote by $\conj{\gamma}$ the
conjugacy class of $\gamma \in \Gamma$ and by $\conj{\Gamma}$ the set
of conjugacy classes of of $\Gamma$. Our first step is to use the
decomposition
\begin{equation}\label{eq.decomp}
 \Hd_q(A \rtimes \Gamma) \, \simeq \, \bigoplus_{\conj{\gamma} \in
   \conj{\Gamma}} \Hd_q(A \rtimes \Gamma)_{\gamma}
\end{equation}
of the Hochschild homology groups of $A \rtimes \Gamma$
\cite{BaumConnes2, BaumConnes1, Brylinski, FeiginTsygan, Karoubi, nistorInvent90}.  
Let $\Prim(\kk)$ be the maximal ideal spectrum of
$\kk$ and denote by $\supp(A) \subset \Prim(\kk)$ the support of
$A$. A first observation is that, if $\gamma \in \Gamma$ is such
$\gamma$ has no fixed points on $\supp(A)$, then the component
corresponding to $\gamma$ in the direct sum decomposition, Equation
\ref{eq.decomp}, vanishes, that is $\Hd_q(A \rtimes \Gamma)_{\gamma} =
0$ for all $q$. Consequently, we also have $\HC_q(A \rtimes
\Gamma)_{\gamma} = \HP_q(A \rtimes \Gamma)_{\gamma} = 0$ for all
$q$. More generally, this allows us to show that the groups $\Hd_q(A
\rtimes \Gamma)_{\gamma} $, $\HC_q(A \rtimes \Gamma)_{\gamma}$, and
$\HP_q(A \rtimes \Gamma)_{\gamma}$ are supported at the ideals
fixed by $\gamma$.

We are especially interested in the case $A = \calo[V]$, the algebra
of regular functions on an affine algebraic variety $V$, in which we
obtain complete results, identifying the periodic cyclic homology
groups with the corresponding orbifold homology groups \cite{Ruan}.
More precisely, we obtain the following result.

\begin{theorem}\label{cor.cross2}\
Let $A$ be a quotient of the ring of polynomials $\CC[X_1, X_2, \ldots
  , X_n]$ and let $\Gamma$ be a finite group acting on $A$ by
automorphisms. Let $\{\gamma_1 , \dots , \gamma_\ell\}$ be a list of
representatives of conjugacy classes of $\Gamma$, let $C_j$ be the
centralizer of $\gamma_j$ in $\Gamma$, and let $V_j \subset V$ be the set of fixed
points of $\gamma_j$ acting on the subset $V \subset \CC\sp{n}$
corresponding to set of maximal ideals of $A$.  Then
\begin{equation*}
  \Hp_q(A \rtimes \Gamma) \ \cong \ \bigoplus_{j=1}^\ell \,
  \bigoplus_{k \in \ZZ} \, H^{q-2k}(V_{j}; \CC)^{C_{j}} \, .
\end{equation*}
\end{theorem}

We refer to \cite{Baranovsky, DolgushevAND, Maszczyk, Posthuma2011} for the
definition of orbifold cohomology groups in the category of smooth
manifolds and to its connections with cyclic homology. We are
interested in the case $A = \calo[V]$ due to its connections with
affine Weyl groups.

For simplicity, we shall consider from now on only complex algebras.
Thus $\kk$ and $A$ are complex vector spaces in a compatible way.  The
paper is organized as follows. In the first section we recall the
definitions of Hochschild and cyclic homologies and of their twisted
versions with respect to the action of an endomorphism. Then we
specialize these groups to cross products and discuss the connection
between the Hochschild and cyclic homologies of a crossed product and
their twisted versions. Our calculations are based on completions with
respect to ideals, so we discuss them in the last subsection of the
first section. The main results are contained in the second
section. They are based on a thorough study of the twisted Hochschild
homology of a finitely generated commutative complex algebra using
Koszul complexes. As an example, in the last section, we recover the 
cyclic homology of certain group algebras of certain affine Weyl groups.

\subsection*{Acknowledgements} 
We thank Mircea Musta\c{t}a, Tomasz Maszczyk, Roger Plymen, and 
Hessel Posthuma for useful discussions.

\section{Basic definitions}

We review in this section several constructions and results needed in
what follows. Thus we introduce the operators $b$ and $B$ needed to
define our various version of the Hochschild and cyclic homologies
that we will use.  We also recall results on the topological (or
$I$-adic) versions of these groups, which are less common and reduce
to the usual definitions when $I = 0$.

\subsection{Hochschild and cyclic homology\label{Sec2}}

Let $A$ be a complex algebra with unit and denote by $A^{\opp}$ the
algebra $A$ with the opposite multiplication and $A^e := A \otimes
A^{\opp}$, so that $A$ becomes a left $A^e$--module with $(a_0 \otimes
a_1) \cdot a := a_0 a a_1$.

Let $g$ be an endomorphism of $A$ and define for $x = a_0
\otimes a_1 \otimes \ldots \otimes a_n \in A\sp{\otimes n+1}$
\begin{equation}\label{eq.def.bg}
\begin{gathered}
    b_g\rp(x) :\, = \, a_0 g(a_1) \otimes a_2 \otimes \ldots \otimes
    a_n + \, \sum_{i=1}^{n-1} \, (-1)^i a_0 \otimes \ldots \otimes a_i
    a_{i+1} \otimes \ldots \otimes a_n \\ \mbox{and } \ \ b_g(x) \, :=
    \, b_g\rp (x) + \, (-1)\sp{n} a_n a_0 \otimes a_1 \otimes \ldots
    \otimes a_{n-1}\,.
\end{gathered}
\end{equation}
Let us denote by $\calb_q\rp(A) := A \otimes (A/\CC 1)\sp{\otimes q}
\otimes A$ and $\calb_q(A) := A \otimes (A/\CC 1)\sp{\otimes q}$, $q
\ge 0$.  Then $b_g\rp$ and $b_g$ descend to maps $b_g\rp :
\calb_q\rp(A) \to \calb_{q-1}\rp(A)$ and $b_g : \calb_q(A) \to
\calb_{q-1}(A)$.  As usual, when $g$ is the identity, we write simply
$b$ instead of $b_g$ and $b\rp$ instead of $b_g\rp$.  We then define

\begin{definition}\label{def.H.twisted}
The {\em $g$-twisted Hochschild homology groups} $\Hd_*(A, g)$ of $A$
are the homology groups of the complex $(\calb_*(A), b_g) =
(\calb_q(A), b_g)_{q \ge 0}$.
\end{definition}

Thus, when $g$ is the identity, and hence $b_g = b$, we obtain that
$\Hd_*(A, g)$ are simply the usual Hochshild homology $\Hd_*(A)$
groups of $A$.

Let $A_g$ be the $A\sp{e}$ {\em right-module} $a(a_1 \otimes a_2) :=
a_2 a g(a_1)$. By tensoring the complex $(\calb_q\rp (A), b\rp)_{q \ge
  1}$ with $A_g$ over $A\sp{e}$ we obtain the complex $(\calb_q(A),
b_g)_{q \ge 0}$.  Since the complex $(\calb_q \rp (A), b')$ is
acyclic--and thus provides a projective resolution of $A$ with $A^e :=
A \otimes A^{\opp}$ modules--we obtain \cite{Brylinski, FeiginTsygan,
  LodayBook, nistorInvent90, Stefan}

\begin{proposition}\label{prop.isom.gt} 
For every $q \ge 0$, we have a natural isomorphism
\begin{equation}
  \Hd_q(A, g) \, \simeq \, \Tor_q^{A^e}(A_g,A) \;.
\end{equation}
\end{proposition}

See \cite{BohmStefan, gabriel13} for some generalizations.
Let us assume now that $A$ is an algebra over a (commutative) ring
$\kk$.  Then we see by a direct calculation that the differential
$b_g$ is $\kk$--linear, and hence each $\Hd_q(A, g)$ is naturally a
$\kk$--module.  Moreover, both $A$ and $A_g$ are $\kk \otimes
\kk$-modules:
\begin{equation}\label{eq.def.tms}
 (z_1 \otimes z_2) a \, := \, g(z_1) z_2 a \, , \ \mbox{ for } z_1,
  z_2 \in \kk \mbox{ and } a \in A_g \,.
\end{equation}
Consequently, $\Tor_q^{A^e}(A_g,A)$ is naturally a $\kk \otimes
\kk$-module, with the action of $z_1 \otimes z_2 \in \kk \otimes \kk$
on either component of the $\Tor$-group being the same.  Examining the
proof of Proposition \ref{prop.isom.gt}, we see that the isomorphism
$\Hd_q(A, g) \, \simeq \, \Tor_q^{A^e}(A_g,A)$ of that proposition is
compatible with these module structures. More related information is
contained in the following results. We first need to recall some basic
definitions.

Let $R$ be a ring (all our rings have units). We assume $\CC \subset
R$, for simplicity (that\rp s the only case we need anyway). We denote
by $\widehat R$ the set of maximal ideals $\mfk m$ of $R$ with the
Jacobson topology.  We assume that $R/\mfk m \simeq \CC$ as
$\CC$-algebras for any maximal ideal ${\mfk m}$.  Recall that if $M$
is a module over a ring $R$ (all our rings have units), then the {\em
  support} of $M$ is the set of maximal ideals $\mfk m \in \widehat R$
such that the localization $M_{\mfk m} := A_{\mfk m} M$ is
nonzero. Thus we have that $M_{\mfk m} = 0$ if, and only if, for every
$m \in M$, there exists $x \notin \mfk m$ such that $x m = 0$.

Returning to our setting, let $\mfk m, \mfk n \in \Spec(\kk)$ be
maximal ideals and $\chi_{\mfk m} : \kk/\mfk m \to \CC$ and
$\chi_{\mfk n} : \kk/{\mfk n} \to \CC$ be the canonical
isomorphisms. Then $({\mfk m}, {\mfk n}) \in \Spec(\kk) \times
\Spec(\kk)$ corresponds to the morphism $\chi_{\mfk m} \otimes
\chi_{\mfk n} \to \kk \otimes \kk \to \CC$. This identifies the
maximal ideal spectrum of $\kk \otimes \kk$ with $\Spec(\kk) \times
\Spec(\kk)$.

\begin{lemma}\label{lemma.Ag} 
The support of $A_g$ in $\Spec(\kk) \times \Spec(\kk)$ is contained in
the set \begin{equation*}\{ \, (g^{-1}(\mfk m), \ \mfk m), \, \mfk m
  \in \widehat \kk \, \} \,.\end{equation*}
\end{lemma}

\begin{proof}
Assume that $({\mfk m}, {\mfk n}) \in \Spec(\kk) \times \Spec(\kk)$ is
such that $g^{-1}({\mfk n}) \neq \mfk m$.  Since $g\sp{-1}({\mfk n})
\neq \mfk m$, we have that $\chi_{\mfk m} \neq \chi_{\mfk n} \circ
g$. Then there exists $z \in \kk$ such that $\chi_{\mfk m}(z) \neq
\chi_{\mfk n}(g(z))$. Hence $w := z \otimes 1 - 1 \otimes g(z)$
satisfies $\chi_{\mfk m} \otimes \chi_{\mfk n}(w) = \chi_{\mfk m}(z) -
\chi_{\mfk n}(g(z)) \neq 0$ is not in $(\mfk m, {\mfk n})$.  However,
if $a \in A_g$, then $wa = g(z)a - g(z)a = 0$. By definition, this
means that the localization of $A_g$ at $({\mfk m}, {\mfk n})$ is
zero, and hence $({\mfk m}, {\mfk n})$ is not in the support of $A_g$.
\end{proof}

Recall \cite{BrylinskiLoc,KazhdanNistorSchneider} that if $S \subset
\kk \otimes \kk$ is a multiplicative subset and $M$ and $N$ are two
bimodules (or $A\sp{e}$ left-modules), then
\begin{equation}\label{eq.localization}
 S\sp{-1}\Tor_q\sp{A\sp{e}}(M, N) \, \simeq\, 
 \Tor_q\sp{A\sp{e}}(S\sp{-1}M, S\sp{-1}N) \, \simeq\, 
 \Tor_q\sp{S\sp{-1}A\sp{e}}(S\sp{-1}M, S\sp{-1}N) \,.
\end{equation}
This shows that the support of $\Tor_q\sp{A\sp{e}}(M, N)$ is contained
in the intersection of the supports of $M$ and $N$.

Let us assume for the rest of this section that the given endomorphism
$g$ of $A$ has a counterpart in an endomorphism of $\kk$--also denoted
$g$--such that $g(z a) = g(z) g(a)$, for all $z \in \kk$ and $a \in
A$.

\begin{corollary}\label{cor.support}
Let $S \subset \kk$ be a $g$-invariant multiplicative subset.  We have
\begin{equation*}
 \Hd_q(S\sp{-1}A, g) \simeq S\sp{-1}\Hd_q(A, g)\,.
\end{equation*}
Hence the support of $\Hd_*(A, g)$ is contained in the set
 \begin{equation*}\
   \{\, ({\mfk m}, {\mfk m}),\ {\mfk m} \in \Spec(\kk), \,
   g\sp{-1}({\mfk m}) = {\mfk m} \, \} \,.
 \end{equation*}
\end{corollary}

\begin{proof}
The first part of the proof follows from Equation
\eqref{eq.localization} for the multiplicative subset $S \otimes S
\subset \kk \otimes \kk$.  (A direct proof can be obtained as in
\cite{KazhdanNistorSchneider} for example).  The support of $A$ is
contained in the diagonal $\{({\mfk m}, {\mfk m}),\, {\mfk m} \in
\Spec(\kk)\}$. The support of $A_g$ is contained in the set
$\{(g^{-1}({\mfk m}), {\mfk m}),\, x \in \Spec(\kk)\}$. Since
$\Hd_q(A, g) \, \simeq \, \Tor_q^{A^e}(A,A_g)$ and
\begin{equation*}
   \{({\mfk m}, {\mfk m})\} \cap \{(g^{-1}({\mfk m}), {\mfk m})\} \, =
   \, \{({\mfk m}, {\mfk m}),\, g\sp{-1}({\mfk m}) = {\mfk m} \} \,
   \subset \, \Spec(\kk) \times \Spec(\kk)\,,
\end{equation*}
the result follows.
\end{proof}

Of the two $\kk$-module structures on $A$ and $A_g$, we shall now
retain just the one given by multiplication to the left, that is, the
usual $\kk$-module structure. We also have the following.

\begin{proposition}\label{prop.support}
The support of $\Hd_q(A, g)$ as a $\kk$-module is
\begin{equation*}
  \{ \, \mfk m \in \Spec(\kk), \ g\sp{-1}(\mfk m) = \mfk m \,
  \}.
\end{equation*}
\end{proposition}

\begin{proof}
Let $\kk_1$ and $\kk_2$ be finitely generated commutative complex
algebras with unit and let $M$ be a $\kk_1 \otimes \kk_2$-module with
support $K \subset \Spec(\kk_1 \otimes \kk_2) = \Spec(\kk_1) \times
\Spec(\kk_2)$.  Then the support of $M$ as a $\kk_1$ module is the
projection of $K$ onto the first component. Combining this fact with
Corollary \ref{cor.support} yields the desired result.
\end{proof}

This proposition then gives the following standard consequences

\begin{corollary}\label{cor.zero}
If $\mfk n \in \Spec(\kk)$ is such that $g\sp{-1}(\mfk n) \neq \mfk
n$, then $\Hd_q(A, g)_{\mfk n} = 0$.
In particular, if $g\sp{-1}(\mfk m) \neq \mfk m$ for all maximal
ideals $\mfk m$ of $\kk$, then $\Hd_q(A, g) = 0$.
\end{corollary}

\begin{proof}
Since the support of $\Hd_q(A, g)$ is the set $\{\mfk m \in
\Spec(\kk), g\sp{-1}(\mfk m) = \mfk m\}$, we obtain that $\mfk n$ is
not in the support of $\Hd_q(A, g)$, that is, $\Hd_q(A, g)_{\mfk n} =
0$.  This proves the first part. The second part follows from the fact
that a $\kk$-module with empty support is zero.
\end{proof}

In order to introduce the closely related cyclic homology groups, let
us first recall the following standard operators acting on
$A\sp{\otimes n}$, $n \ge 0$, \cite{ConnesIHES}
\begin{equation}\label{equation.B}
\begin{split}
 s(\Tt) \, & := \, 1\otimes \Tt \, ,  \\
 t_g(\Tt) \, & := \, (-1)^n g(a_n) \otimes a_0\otimes\ldots\otimes
 a_{n-1} \, , \mbox{ and } \\
  B_g(\Tt) \, & := \ s \, \sum_{k=0}^{n} t_g^k(\Tt) \, .
\end{split}
\end{equation}

Let $\gamma$ act diagonally on $\calb_q(A)$:\ $g(\Tt) = g(a_0) \otimes
g(a_1) \otimes \ldots \otimes g(a_n)$.  Then $B_g$ descends to
differential
\begin{equation*}
 B_g : X_q(A) := \calb_q(A)/(g - 1) \calb_q(A) \to X_{q+1}(A)
\end{equation*}
satisfying $B_g\sp{2} = 0$ and $B_g b_g + b_g B_g = 0$.  Therefore $(
X_q(A) , b_g, B_g)$ is a mixed complex \cite{JonesKassel89}, where
$b_g$ is the $g$-twisted Hochschild homology boundary map defined in
Equation \eqref{eq.def.bg}.  Recall how its cyclic homology are
computed. Let
\begin{equation*}
  {\mathcal C}_n(A) \ := \ \bigoplus_{q\geq 0} \, X_{n-2q}(A) \,.
\end{equation*}
Then the {\em $g$-twisted cyclic homology} groups of the unital
algebra $A,$ denoted $\Hc_n(A,g),$ are the homology groups of the
cyclic complex $({\mathcal C}_*(A), b_g+B_g)$ \cite{ConnesIHES}. The
homology groups of the 2-periodic complex
\begin{equation*}
 b_g + B_g \, : \, \prod_{q \in \ZZ} X_{i
  +2q} (A)\to \prod_{q \in \ZZ} X_{i +1 +2q}(A)\,, \quad i \in \ZZ/2\ZZ\,,
\end{equation*}
are the {\em $g$-twisted periodic cyclic homology} groups of the
unital algebra $A$ and are denoted $\HP_n(A,g)$.

\subsection{Crossed-products}

Let in this subsection $A$ be an arbitrary complex algebra and
$\Gamma$ be a group acting on $A$ by automorphisms:\ $\alpha : \Gamma
\to \Aut(A)$. We do not assume any topology on $\Gamma$, that is,
$\Gamma$ is {\em discrete}. Then $A \rtimes \Gamma$ is generated by $a
\gamma$, $a \in A$, $\gamma \in \Gamma$, subject to the relation $a
\gamma b \gamma^\prime = a \alpha_{\gamma}(b) \gamma \gamma^\prime$.

We want to study the cyclic homology of $A \rtimes \Gamma$.  Both the
cyclic and Hochschild complexes of $A \rtimes \Gamma$ decompose in
direct sums of complexes indexed by the conjugacy classes
$\conj{\gamma}$ of $\Gamma$. Explicitly, to $\conj{\gamma} := \{ g
\gamma g\sp{-1}, g \in \Gamma \}$ there is associated the subcomplex
of the Hochschild complex $(\calb_q(A \rtimes \Gamma), b)$ generated
linearly by tensors of the form $a_0\gamma_0 \otimes a_1 \gamma_1
\otimes \ldots \otimes a_k \gamma_k$ with $\gamma_0 \gamma_1 \ldots
\gamma_k \in \conj{\gamma}$. The homology groups of this subcomplex of
the Hochschild complex $(\calb_q(A \rtimes \Gamma), b)$ associated to
$\conj{\gamma}$ will be denoted by $\Hd_*(A \rtimes
\Gamma)_{\gamma}$. We similarly define $\HC_*(A \rtimes \Gamma)_{\gamma}$ and
$\HP_*(A \rtimes \Gamma)_{\gamma}$, see \cite{FeiginTsygan,
  Karoubi, nistorInvent90} for instance. This yields the decomposition of Equation
\eqref{eq.decomp}. Similarly,
\begin{equation}
 \HC_q(A \rtimes \Gamma) \, \simeq \, \oplus_{\conj{\gamma} \in
   \conj{\Gamma}} \HC_q(A \rtimes \Gamma)_{\gamma}\,.
\end{equation}
If $\Gamma$ has finitely many conjugacy classes, a similar relation
holds also for periodic cyclic homology. Such decompositions hold 
more generally for group-graded algebras \cite{Lorenz}.

Let us assume now that $\Gamma$ is finite. Let, for any $g \in
\Gamma$, $C_g$ denote the centralizer of $g$ in $\Gamma$, that is,
$C_g = \{ \gamma \in \Gamma, g \gamma = \gamma g \}$. Then we have the
following result \cite{BlockGetzler, Brylinski, BrylinskiNistor, 
Farinati, FeiginTsygan, nistorInvent90, Lorenz, posthuma06, SW}.

\begin{proposition}\label{prop.decomp}
Let $\Gamma$ be a finite group acting on $A$ and $\gamma \in \Gamma$.
Then we have natural isomorphisms $ \Hd_q(A \rtimes \Gamma)_{\gamma}
\simeq \Hd_q(A, \gamma)^{C_\gamma}$, $\Hc_q(A \rtimes \Gamma)_{\gamma}
\simeq \Hc_q(A, \gamma)^{C_\gamma}$, and $\HP_q(A \rtimes
\Gamma)_{\gamma} \, \simeq \, \HP_q(A, \gamma)^{C_\gamma}$.
\end{proposition}

An approach to this result (as well as to the spectral sequence for
the case when $\Gamma$ is not necessarily finite), using cyclic objects
and their generalizations, is contained in \cite{nistorInvent90}.
This proposition explains why we are interested in the twisted
Hochschild homology groups. The same proof will apply in the case of
algebras endowed with filtrations.  A proof of this proposition in the
more general case of filtered algebras is provided in Proposition
\ref{prop.decomp2}. Algebras with these properties are called {\em
  topological algebras} in \cite{Hubl1}, but this terminology
conflicts with the terminology used in other papers in the field.

\subsection{Finite type algebras and completions}
We continue to assume that $A$ is a $\kk$--algebra with unit, for some
commutative ring $\kk$, $\CC \subset \kk$.  We shall need the
following definition from \cite{KazhdanNistorSchneider}

\begin{definition}
A $\kk$-algebra $A$ is called a {\em finite type $\kk$-algebra} if
$\kk$ is a finitely generated ring and $A$ is a finitely generated
$\kk$-module.
\end{definition}

We shall need to consider completions of our algebras and complexes
with respect to the topology defined by the powers of an ideal. Let us
consider then a vector space $V$ endowed with a decreasing filtration
\begin{equation*}
  V \, = \, F_0V \supset F_1V\supset \ldots \supset F_nV \supset
  \ldots\; .
\end{equation*}
Then its {\em completion} is defined by
\begin{equation}
  \widehat{V} \, = \, \displaystyle{\lim_{\longleftarrow}}\, V/F_j V
  \;.
\end{equation}
If the natural map $V \to \widehat{V}$ is an isomorphism, we shall say
that $V$ is {\em complete}.  A linear map $\phi:V \to W$ between two
filtered vector spaces is called {\em continuous} if for every integer
$n \ge 0$ there is an integer $k \ge 0$ such that $\phi(F_kV) \subset
F_{n}W$. It is called {\em bounded} if there is an integer $k$ such
that $\phi(F_nV) \subset F_{n-k}W$ for all $n$. Clearly, every bounded
map is continuous.  A bounded map $\phi : V \to W$ of filtered vector
spaces extends to a linear map $\widehat \phi :\widehat{V} \to
\widehat{W}$ between their completions.  If $V$ and $W$ are filtered
vector spaces, the completed tensor product is defined as
\begin{equation}
 V \widehat \otimes W \, := \, \displaystyle{\lim_{\longleftarrow}}\;
 V/F_j V \otimes W/F_j W \;.
\end{equation}

We shall need the following standard lemma

\begin{lemma}\label{lemma.qi.complete}
 Let $f_* : (V_*, d) \to (W_*, d)$, $* \ge 0$, be a morphism of
 filtered complexes with $f_q(F_n V_q) \subset F_n W_q$ for all $n, q
 \ge 0$. We assume that all groups $V_q$ and $W_q$, $q \ge 0$, are
 complete and that $f_*$ induces isomorphisms $H_q(F_nV/F_{n+1}V) \to
 H_q(F_nW/F_{n+1}W)$ for all $n \ge 0$ and $q \ge 0$. Then $f$ is a
 quasi-isomrphism.
\end{lemma}

\begin{proof}
 A spectral sequence argument (really, just the \dlp Five Lemma\drp)
 shows that $f_*$ defines a quasi-isomorphism $(V_*/F_nV_*, d) \to
 (W_*/F_n W_*, d)$, for all $n \ge 0$.  The assumption that the
 modules $V_q$ and $W_q$ are complete and an application of the
 $\displaystyle{\lim_{\leftarrow}}\sp{1}$-exact sequence for the
 homology of a projective limit of complexes then give the result.
\end{proof}

Typically, the filtered vector spaces $V$ that we will consider will
come from $A$-modules endowed with the $I$-adic filtration
corresponding to a two-sided ideal $I$ of $A$. More precisely, $F_n V
:= I ^nV$, for some fixed two-sided ideal $I$ of $A$.  Most of the
time, the ideal $I$ will be of the form $I = I_0 A$, where $I_0
\subset \kk$ is an ideal. Consequently, if $M$ and $N$ are a right,
respectively a left, $A$--module endowed with the $I$--adic
filtration, then we define
\begin{equation}
  M \widehat \otimes_A N \, := \, \displaystyle{\lim_{\longleftarrow}}\; 
  M/I^nM \otimes_A N/I^nN \;.
\end{equation}

Basic results of homological algebra extend to $A$--modules endowed
with filtrations if one is careful to use only {\em admissible}
resolutions, that is resolutions that have {\em bounded} $\CC$--linear
contractions, in the spirit of relative homological algebra. The
completion of every admissible resolution by finitely generated free
modules is admissible, and this is essential for our argument.

Let $A$ be a filtered algebra by $F_q A \subset A$, where $F_q A F_p A
\subset F_{p + q} A$. We then define $A^{\widehat\otimes q} :=
\displaystyle{\lim_{\longleftarrow}}\, (A/ F_nA)^{\otimes q}$ be the
topological tensor product in the category of complete modules.  In
the case of the $I$-adic filtration, we have $F_p A := I\sp{p}A$ and
$A^{\widehat\otimes q} := \displaystyle{\lim_{\longleftarrow}}\, (A/
I^nA)^{\otimes q}$.  Let $g$ be an endomorphism $g : A \to A$ as
before, and assume also that $g$ preserves the filtration and that it
induces an endomorphism of $\kk$ as well. (So $g$ is an endomorphism
of $A$ as a $\CC$-algebra, not as a $\kk$-algebra.)  Then the {\em
  topological, twisted Hochschild complex} \cite{Seibt1} of $A$ is
denoted $(\widehat{\calb}_*(\widehat A), b_g)$ and is the completion
of the usual Hochschild complex $(\calb_*(A), b_g)$ with repect to the
filtration topology topology. See also \cite{Hubl1}. Explicitly,
\begin{equation}
  \widehat A \stackrel{b_g}{\longleftarrow} A\widehat \otimes A
  \stackrel{b_g}{\longleftarrow} A^{\widehat\otimes 3}
  \stackrel{b_g}{\longleftarrow} \ldots \stackrel{b_g}{\longleftarrow}
  A^{\widehat\otimes n+1} \stackrel{b_g}{\longleftarrow} \ldots
\end{equation}
We shall denote by $\tHd_*(\widehat A, g)$ the homology of the
topological, twisted Hochschild complex. Note that $\Hd_q(\widehat
A,g)$ is naturally a $\kk$-module since $b_g$ is $\kk$-linear. When
$I=0$, we recover, of course, the usual Hochschild homology groups
$\Hd_*(A,g)$ and we have natural, $\kk$-linear maps $\Hd_q(A, g) \to
\tHd_*(\widehat A, g)$.  The topological, $b'$--complex
$(\widehat{\calb}_*(\widehat A), b')$ is defined analogously. We then
see that $s$ is a bounded contraction for $(\widehat{\calb}_*(\widehat
A), b')$, where, we recall $s(\Tt)=1\otimes \Tt$. In particular, the
complex so $(\widehat{\calb}_*(\widehat A), b')$ is still acyclic.

From now on, we shall fix an ideal $I$ of $A$ and consider only the
$I$-adic filtration given by $F_p A = I\sp{p} A$.  We will see than
that both the Hochschild homology and the periodic cyclic homology
groups behave very well under completions for the $I$-adic topologies.
We need however to introduce some notation and auxiliary material for
the following theorem.

We shall need the following notation.  Let $I = I_0 A$, where $I_0$ is
an ideal of $\kk$. We denote by $\widehat{\kk}$ the completion of $\kk$
with respect to the topology defined by the powers of $I_0$, that is,
$\widehat{\kk} := \displaystyle{\lim_{\longleftarrow}} \, \kk/I_0^n$.
Then $\widehat A \simeq A \otimes_{\kk} \widehat{\kk}$, by standard
homological algebra \cite{AtiyahMacDonald, BourbakiAlgComm}.
Similarly, we shall denote by $\widehat A^{e} := \widehat A \widehat
\otimes \widehat A^{\opp}$, the completion of $A^{e}$ with respect to
the topology defined by the ideal $I_2 := I_0 \otimes \kk + \kk
\otimes I_0$. Let $g$ be a $\kk$-algebra endomorphism of $A$ and let
$A_g$ be the $A\sp{e}$ right-module with the action $a (a'\otimes a'')
= a''a g(a')$.

The following result was proved for $g = 1$ (identity automorphism)
in \cite{KazhdanNistorSchneider}.

\begin{theorem}\label{theorem.completion} \
Let $A$ be a finite-type $\kk$-algebra.  Using the notation just
introduced, we have that $\Hd_q(A, g)$ is a finitely generated
$\kk$-module for each $q$. If we endow $A$ with the filtration defined
by the powers of $I$, then the natural map $\Hd_*(A, g) \to
\tHd_*(\widehat A, g)$ and the $\kk$--module structure on $\Hd_*(A)$
define a $\kk$-module isomorphism
\begin{equation*} 
\Hd_*(A, g) \otimes_\kk \widehat{\kk} \simeq \tHd_*(\widehat A, g).
\end{equation*} 
\end{theorem}

\begin{proof} 
We consider a resolution of $A$ by finitely-generated, free, left
$A^e$ modules say $(A^e\otimes V_i,d_i)$, which always exist since
$A^e$ is Noetherian. By tensoring this resolution with $A_g$ over
$A\sp{e}$, we obtain that the homology groups $\Hd_q(A, g) \simeq
\Tor_q^{A^e}(A_g, A)$ are all finitely generated $\kk$-modules. Since
completion over Noetherian rings is exact in the category of finitely
generated $\kk$-modules, the above complex completed by $I_2 := I_0
\otimes \kk + \kk \otimes I_0$ provides an admissible resolution of
$\widehat{A}$ namely $(\widehat{A}^e\otimes
V_i,\widehat{d}_i)$. (Recall that by the admissibility of a complex we
mean the existence of a bounded contraction, this property in this
case is provided by \cite{KazhdanNistorSchneider}.)  Now $
\mathrm{Tor}^{\widehat{A}^e}_*(\widehat{A}_g,\widehat{A})$ is the
homology of the complex,
\begin{eqnarray*}
  \widehat{A}_g \otimes _{\widehat{A}^e} (\widehat{A}^e\otimes V_i,\widehat{d}_i)
  =(\widehat{A}_g\otimes V_i, 1\otimes \widehat{d}_i).
\end{eqnarray*}
The right-hand side is the completion  with respect to the  ideal $I$
of the complex $(A_g\otimes V_i,1\otimes d_i)$, a complex   
whose homology is $\Hd_*(A,g)$. Hence 
\begin{equation*}
  \Tor^{\widehat{A}^e}_*(\widehat{A}_g,\widehat{A}) \, = \,
  \Hd_*(A,g)\otimes_{\kk} \widehat{\kk} \,,
\end{equation*} 
Since $\Hd_*(A,g)$ is a finitely generated $\kk$-module.  On the
hand, since the $I$-adic bar-complex
  $(\widehat{\calb}_*(\widehat A), b')$ has an obvious contraction
  that makes it admissible, there exist morphisms
\begin{equation*}(\widehat{A^e}\otimes V_i,\widehat{d}_i) \overset{\phi}{\rightarrow} 
(\widehat{\calb}_*(\widehat A), b')
\overset{\psi}{\rightarrow}(\widehat{A^e}\otimes V_i,\widehat{d}_i),\end{equation*}
with $\phi\circ \psi$, respectively $\psi\circ \phi$, homotopic to
identity on $(\widehat{\calb}_*(\widehat A), b')$, respectively
$(\widehat{A^e}\otimes V_i,\widehat{d}_i)$, thus after tensoring with
$\widehat{A}_g$ over $\widehat{A}^e$, we obtain that
$\Tor^{\widehat{A}^e}_*(\widehat{A}_g,\widehat{A}) =
\Hd_*^{top}(\widehat{A},g)$.
\end{proof}

Using the same method, one can prove also that 
\begin{equation}\label{eq.Hubl}
  \tHd_*(\widehat A, g) \simeq \Tor_q^{\widehat A^e}(\widehat A_g,
  \widehat A) \,,
\end{equation}
where the right hand side denotes the derived functors of the tensor
product of $\widehat A^e$ filtered modules. This result that was
proved by H\umlaut{u}bl in the commutative case, \cite{Hubl1}.  We now
recall some properties of periodic cyclic homology with repect to
completions. We have the following result of Goodwille
\cite{Goodwillie}.

\begin{theorem}[Goodwillie] \label{Goodwillie} \
If $J \subset B$ is a nilpotent ideal of an algebra $B$, then the
quotient morphism $B \to B/J$ induces an isomorphism $\Hp_*(B) \to
\Hp_*(B/J)$.
\end{theorem}

This gives the following result of \cite{Seibt1} (Application 1.15(2)).

\begin{theorem}[Seibt] 
\label{Seibt} Let $J$ be a two--sided ideal of an algebra $B$. 
Then the quotient morphism $B \to B/J$ induces an isomorphism
$\tHp_*(\widehat B) \to \Hp_*(B/J)$, where the completion is with
respect to the powers of $J$.
\end{theorem}

Let us assume that $I \subset A$ is a two-sided ideal invariant for
the action of the finite group $\Gamma$, so that $\widehat A \rtimes
\Gamma$ is defined. Then the Hochschild, cyclic, and periodic cyclic
complexes of the completion of $\widehat A \rtimes \Gamma$ still
decompose according to the conjugacy classes of $\Gamma$.  Recall that
$C_\gamma$ denotes the centralizer of $\gamma \in \Gamma$ in $\Gamma$
(that is, the set of elements in $\Gamma$ commuting with $\gamma$) and
denote $J := I\rtimes \Gamma = I\otimes \CC[\Gamma]$.

\begin{proposition}\label{prop.decomp2}  
Let $\Gamma$ be a finite group, then $\widehat{A}\rtimes \Gamma$ is
naturally isomorphic to the $J := I\rtimes \Gamma$-adic completion of
$A\rtimes \Gamma$ and we have natural isomorphisms $\tHd_q(\widehat A
\rtimes \Gamma)_{\gamma} \simeq \tHd_q(\widehat A,
\gamma)^{C_\gamma}$, $\tHc_q(\widehat A \rtimes \Gamma)_{\gamma}
\simeq \tHc_q(\widehat A, \gamma)^{C_\gamma}$, and, similarly,
$\tHp_q(\widehat A \rtimes \Gamma)_{\gamma} \simeq \tHp_q(\widehat A,
\gamma)^{C_\gamma}$.
\end{proposition}

\begin{proof} 
The first statement is a direct consequence of the definitions.
Observe that $\widehat{\calb}_n(\widehat{A\rtimes \Gamma}) =(A\rtimes
\Gamma)^{\widehat{\otimes}n+1}$ is the same as
$A^{\widehat{\otimes}n+1}\rtimes\Gamma^{n+1}$ and this space admits a
known decomposition into conjugacy classes with respect to the action
of $\Gamma$, namely,
\begin{equation*}
\widehat{\calb}_q(\widehat{A\rtimes \Gamma}) \, = \,
\bigoplus_{\conj{\gamma} \in \conj{\Gamma}}
\widehat{\calb}_q(\widehat{A\rtimes \Gamma})_\gamma,\end{equation*} where
$(a_0,a_1,\ldots,a_n)(g_0,g_1,\ldots,g_n) \in
\widehat{\calb}_*(\widehat{A\rtimes \Gamma})_\gamma$ exactly when $g_0
g_1\ldots g_n\in \langle\gamma\rangle$.  As in Proposition
\ref{prop.decomp}, the chain map,
\begin{equation*}
  (a_0,a_1,\ldots,a_n)(g_0,g_1,\ldots,g_n) \, \mapsto \,
\frac{1}{|C_\gamma|} \sum_{h\in C_\gamma}(\gamma
hg_0^{-1}(a_0),h(a_1), \ldots,hg_1g_2 \cdots g_{n-1}(a_n))
\end{equation*}
now provides an quasi-isomorphism between
$\widehat{\calb}_*(\widehat{A\rtimes \Gamma})_\gamma$ and
$\widehat{\calb}_*(\widehat A,b_{\gamma})^{C_\gamma}$.
\end{proof}

We have the following corollary.

\begin{corollary}\label{cor.comp.loc}
In the framework of Proposition \ref{prop.decomp2}, let us assume
that $\gamma \in \Gamma$ acts freely on $\Spec(\kk)$. Then
\begin{equation*}
  \tHd_q(\widehat A \rtimes \Gamma)_{\gamma} \, = \, \tHd_q(\widehat A
  \rtimes \Gamma)_{\gamma} \, = \, \tHd_q(\widehat A \rtimes
  \Gamma)_{\gamma} \, = \, 0\,.
\end{equation*}
\end{corollary}

\begin{proof}
Corollary \ref{cor.zero} gives that $\Hd_q(A, \gamma) = 0$
for all $q$. Proposition \ref{prop.decomp} then gives that 
$\Hd_q(A \rtimes \Gamma)_{\gamma} = 0$ for all $q$. Connes\rp\ SBI-long
exact sequence relating the cyclic and Hochschild homologies then gives
that $\Hc_q(A \rtimes \Gamma)_{\gamma} = 0$ for all $q$.  This in
turn implies the vanishing of the periodic cyclic homology.
\end{proof}

\section{Crossed products of regular functions}
Our focus from now on will be on the case when $A$ is a finitely
generated commutative algebra and $\Gamma$ is finite. Goodwillie\rp s
theorem then allows us to reduce to the case when $A = \calo[V]$, where $V$ is a
complex, affine algebraic variety and where $\calo[V]$ is the ring of
regular functions on $V$.  We will thus concentrate on crossed
products of the form $\calo[V] \rtimes \Gamma$, where $\Gamma$ is a
discrete group acting by automorphisms on $\calo[V]$. We begin with a
discussion of the case when there is, in fact, no group.

\subsection{The commutative case}
In this section we recall some known constructions and results, 
see \cite{Baranovsky, Brylinski, Emmanouil, Farinati, FeiginTsygan, Kaledin, 
LodayBook, LodayQuillen} 
and the references therein. We mostly follow (and expand) the method in 
\cite{BrylinskiNistor}, where rather the case of $\CI$-functions 
was considered.
We begin by fixing some notation. As usual, we shall denote by
$\Omega^q(X)$ the space of algebraic $q$--forms on a smooth, algebraic
variety $X$. If $Y \subset X$ is a smooth subvariety, we denote by
$\omega_{\vert Y} \in \Omega^q(Y)$ the restriction of $\omega \in
\Omega^q(X)$ to $Y$.

In this subsection, we will consider affine, complex algebraic
varieties $X$ and $V$, where $X$ is assumed to be smooth, $V \subset
X$, and $I \subset \calo[X]$ is the ideal defining $V$ in $X$.  Recall
that $\calo[X]$ denotes the ring of regular functions on $X$, it is
the quotient of a polynomial ring by the ideal defining $X$. Let $g :
\calo[X] \to \calo[X]$ be an endomorphism such that
$g\sp{-1}(I)=I$. We shall denote by $X^{g}$ the set of fixed points of $g : X
\to X$ and assume that it is also a smooth variety. We then let $\chi_g$
be the {\em twisted} Connes-Hochschild-Kostant-Rosenberg map $\chi_g :
\calo[X]^{\otimes n+1} \to \Omega^{n}(X^{g})$,
\begin{equation}\label{eq.Hochschild-Kostant-Rosenberg}
  \chi_g(a_0\otimes \ldots \otimes a_n) \to \frac{1}{n!}
  a_0da_1\ldots da_n \vert_{X^{g}} \,,
\end{equation}
where the restriction to $X\sp{g}$ is defined since $X\sp{g}$ is
  smooth. The map $\chi_g$ descends to a map
\begin{equation}\label{eq.Hochschild-Kostant-Rosenberg&}
  \chi_g : \Hd_n(\calo[X], g) \, \to \, \Omega^{n}(X^{g}) \,, \quad n \ge 0 \; ,
\end{equation}
for any fixed $n$. When $g = 1$, the identity automorphism, the map
$\chi_g$ is known to be an isomorphism for each $n$
\cite{LodayQuillen}. This fundamental result has been proved in
the smooth, compact manifold case in \cite{ConnesIHES}. See also
\cite{Hubl0, Hubl1, Karoubi87, LodayBook}. 

We want to identify the groups $\Hd_*(A, g)$, $A = \calo[X]$, and
their completions with respect to the powers of $I\sp{n}$.  The idea
is to first use localization to reduce to the case when $X$ is a
vector space, $X = \CC^{N}$, and $g$ acts linearly on $\CC^{N}$. In
that case, the calculation is achieved by constructing a more
appropriate free resolution of the module $A_g$ with free $A\sp{e}$
modules using Koszul complexes (here $A = \calo[X]$).

While the proof is rather lenghty, it is standard and follows a well
understood path in algebraic geometry and commutative
algebra. However, this path is not so well known and understood when
it comes to cyclic homology calculations, so we include all the
details for the benefit of the reader.  The reader with more
experience (or less patience), can skip directly to Proposition
\ref{prop.brylinski}.  Our first step is to construct the relevant
Koszul complex.


\begin{definition}\label{def.koszul}
Let $R$ be a commutative complex algebra, $M$ be an $R$-module, $E$ be
a finite dimensional complex vector space, and $f: E \to R$ be a
linear map. Then to this data we associate the Koszul complex $(\KK_*,
\pa) = (\KK_*(M, E, f), \pa)$, with differential $\partial:\KK_j\to
\KK_{j-1}$, defined by
\begin{equation*}
\begin{gathered}
  \KK_j = \KK_j(M, E, f) := M \otimes \wedge^j E \ \mbox{
    and}\\ \partial ( m \otimes e_{i_1} \wedge \ldots \wedge e_{i_j} )
  \, = \, \sum_{k=1}^j (-1)^{k-1} f(e_{i_k}) m \otimes e_{i_1} \wedge
  \ldots \wedge \widehat{e}_{i_k}\wedge \ldots \wedge e_{i_j} \,,
\end{gathered}
\end{equation*}
where $m \in M$ and $e_1,\ldots,e_n$ is an arbitrary basis of $E$.
\end{definition}

In the rest of this subsection, by $E$ and $F$, sometimes decorated
with indices, we shall denote {\em finite-dimensional, complex
vector spaces.}
The definition above is correct as explained in the following remark.

\begin{remark}\label{rem.diff.koszul}
The definition of the differential in the Koszul comples is, in fact,
independent of the choice of a basis of $E$. Indeed, if we denote by
$i_\xi : \wedge^j E \to \wedge^{j-1} E$ the contraction with $\xi \in
E\sp{*}$, then $\pa = \sum_{j} f(e_j) \otimes i_{e_j\sp{*}}$, where
$e_j\sp{*}$ is the dual basis of $e_j$, and this definition is
immediately seen to be independent of the choice of the basis $(e_j)$
of $E$.
\end{remark}

We need to look first at the simplest instances of Koszul complexes.

\begin{example}\label{ex.dim.one}
Let $E = \CC$. Then $f : E \to R$ is completely determined by $r :=
f(1) \in R$. Then, up to an isomorphism, $\KK_*(R, \CC, f)$ is
isomorphic to the complex
\begin{equation*}
 0 \longleftarrow M \overset{r}\longleftarrow M \longleftarrow 0
 \longleftarrow 0 \ \ldots
\end{equation*}
that has only two non-zero modules:\ $\KK_0 \simeq \KK_1 \simeq M$. We
are, in fact most interested in the case $M = R = \calo[\CC] \simeq
\CC[x]$ (we identify the ring of polynomial functions on $\CC$ with the
ring of polynomials in one indeterminate $X$) and $r = f(1) = x$, in
which case the Koszul complex $\KK_*(\CC[x], \CC, x)$ provides a free
resolution of the $\CC[x]$ module $\CC \simeq \CC[x]/x\CC[x]$ (thus on
$\CC$ we consider the $\CC[x]$-module structure given by \dlp
evaluation at 0\drp, that is, $P \cdot \lambda = P(0) \lambda$ for $P
\in \CC[x]$ and $\lambda \in \CC$). The other case in which we are
interested is the complex $\KK_*(\CC[x], \CC, 0)$, in which case, the
homology coincides with $\KK_*(\CC[x], \CC, x)$ itself, because the
differential is zero.
\end{example}

We treat the other Koszul complexes that we need by combining the
two basic examples considered in Example \ref{ex.dim.one} above,
using also the functoriality of the Koszul complex. In all our applications,
we will have $M = R$, so we assume from now on that this is the case.
The next remark deals with functoriality.

\begin{remark}\label{rem.functoriality}
First, a ring morphism $\phi: R \to R_1$ induces a morphism of
complexes
\begin{equation*}
 \phi_* \, : \, \KK_*(R, E, f) \, \to \, \KK_*(R_1, E,\phi\circ f)
\end{equation*}
given by the formula $\phi_*(r \otimes e_{i_1} \wedge \ldots \wedge
e_{i_j}) = \phi(r) \otimes e_{i_1} \wedge \ldots \wedge e_{i_j}$.
Then we notice also that a linear map $g : E_1 \to E$ induces a
natural morphism of complexes
\begin{equation*}
 g_* : \KK_*(R, E_1, f \circ g) \to \KK_*(R, E, f)
\end{equation*}
given by the formula $g_*(r \otimes e_{i_1} \wedge \ldots \wedge
e_{i_j}) = r \otimes g(e_{i_1}) \wedge \ldots \wedge g(e_{i_j})$. If
$g$ is an isomorphism, then $g_*$ is an isomorphism as well.
\end{remark}

Let us denote by $\overset{.}{\otimes}$ the graded tensor product of
complexes. Koszul complexes can be simplified using the following well
known lemma whose proof is a direct calculation.

\begin{lemma}\label{one}
Assume that $R = R_1 \otimes R_2$, $E = E_1 \oplus E_2$, $f_i : E_i \to R_i$
linear, and $f = f_1 \otimes 1 + 1 \otimes f_2 := ( f_1 \otimes 1 , 1 \otimes f_2 )$. 
Then
\begin{equation*}
 \KK_*(R, E, f) \, \simeq \, \KK_*(R_1, E_1, f_1) \overset{.}{\otimes}
 \KK_*(R_2, E_2, f_2) \,.
\end{equation*}
\end{lemma}

Let $E$ be a vector space and let $E^*$ be its dual. We denote by
$\calo[E]$ the ring of regular functions on $E$, as usual. In this
case, $\calo[E]$ is isomorphic to the ring of polynomial functions in
$\dim(E)$ variables. Then there is a natural inclusion
\begin{equation}
 i: E^*\to \calo[E] \,.
\end{equation}
Although we shall not use that, let us mention, in order to provide
some intuition, that the differential of the resulting Koszul complex
$\KK_*(\calo[E],E^*,i)$ is the Fourier transform of the standard de
Rham differential. Recall that in this
subsection $E$ and $F$ (sometimes with indices)
denote finite-dimensional, complex vector spaces.
Its homology is given by the following lemma.

\begin{lemma}\label{two}
Let  $i : E\sp{*}
\to \calo[E]$ be the canonical inclusion. Then the resulting complex
$\KK_*(\calo[E], E\sp{*}, i)$ has homology
\begin{equation}
  H_q \KK_*(\calo[E], E\sp{*}, i) \ \simeq \
  \begin{cases}
     \ 0 & \mbox{ if } \    q >0 \\                    
     \, \CC & \mbox{ if }  \   q = 0 \,,                   
  \end{cases}
\end{equation}
with the isomorphism for $q = 0$ being given by the the evaluation at
$0 \in E$ morphism $\KK_0(\calo[E], E\sp{*}, i) = \calo[E] \to \CC$.
\end{lemma}

\begin{proof} 
We can assume that $E = \CC\sp{n}$, by the functoriality of the Koszul
complexes (Remark \ref{rem.functoriality}).  If $n = 1$, we may
identify $\calo[E] = \CC[x]$, in which case the result has already
been proved (see Example \ref{ex.dim.one}).  In general, we can
proceed by induction by writting $\CC\sp{n} = \CC\sp{n-1} \oplus \CC$
and using the fact that the Koszul complex associated to $E =
\CC\sp{n}$ is the tensor product of the Koszul complexes associated to
$\CC\sp{n-1}$ and $\CC$ and then invoquing Lemma \ref{one}.  The
calculation of the homology groups is completed using the
K\umlaut{u}nneth formula, which gives that the homology of a tensor
product of complexes (over $\CC$) is the tensor product of the
individual homologies \cite{MacLane}. 
See also \cite{GelfandManinBook}.
\end{proof}

It is useful to state explicitly the following direct consequence of
the above lemma and of the functoriality of the Koszul complexes
(Remark \ref{rem.functoriality}).

\begin{corollary}\label{cor.two}
Let $R = \calo[F]$ 
and let $g : E \to F\sp{*}$ be an isomorphism. Let $R_1$ be a
commutative complex algebra.  Denote by $f := i \circ g : E \to
\calo[F]$, then 
\begin{equation}
  H_q \KK_*(R \otimes R_1, E, f) \ \simeq \
  \begin{cases}
     \ 0 & \mbox{ if } \    q >0 \\                    
     \, R_1 & \mbox{ if }  \   q = 0 \,,                   
  \end{cases}
\end{equation}
with the isomorphism for $q = 0$ being induced by the the evaluation
at $0 \in F$ morphism $\KK_0(R \otimes R_1, E, f) = R \otimes R_1 =
\calo[F] \otimes R_1 \to \CC \otimes R_1 = R_1$.
\end{corollary}

Most of the time, the above corollary will be used for $R_1 = \CC$.
We now prove the result that we need for K\umlaut{u}nneth complexes.

\begin{proposition}\label{prop.res}
Let $h : E_1 \to E\sp{*}$ be a linear map and $f := i \circ h : E_1
\to \calo[E]$, where $i : E\sp{*} \to \calo[E]$ is the canonical
inclusion.  Denote by $F \subset E$ the annihilator of $h(E_1) \subset
E\sp{*}$ and by $F_1 \subset E_1$ the kernel of $h$.  Then the
restriction morphism $\calo[E] \to \calo[F]$ and the choice of a
projection $E_1 \to F_1$ define isomorphisms
\begin{equation*}
    res \, : \, H_q \KK_*(\calo[E], E_1, f) \, \simeq \, H_q
    \KK_*(\calo[F], F_1, 0) \, = \, \calo[F] \otimes \wedge^q F_1 \;.
\end{equation*}
These isomorphisms are independent of the choice of the projection
$E_1 \to F_1$.
\end{proposition}

\begin{proof}
Let $F_1\rp$ be the kernel of the chosen projection $E_1 \to F_1$,
which thus gives a direct sum decomposition $E_1 = F_1 \oplus
F_1\rp$. Similarly, let us chose a complement $F\rp$ to $F$ in $E$,
which also gives a direct sum decomposition $E = F \oplus F\rp$. Then
we obtain that $\calo[E] = \calo[F] \otimes \calo[F\rp]$ and that the
map $f$ splits as $f_1 \otimes 1 + 1 \otimes f_2$, $f_1 : F_1 \to \calo[F]$,
$f_2 : F_1\rp \to \calo[F\rp]$, with $f_1 = 0$
and with $f_2$ obtained from a linear isomorphism $F_1\rp \to (F\rp)\sp{*}$.
Then Lemma \ref{one} gives that $ \KK_*(\calo[E], E_1, f) \, \simeq \,
\KK_*(\calo[F], F_1, 0) \overset{.}{\otimes} \KK_*(\calo[F\rp],
F_1\rp, f_2)$. The homology of the last complex is given by Corollary
\ref{cor.two} with $R_1 = \CC$ and is seen to be $\CC$ in dimension
zero and zero in the other dimensions. This gives the desired result
since $\KK_*(\calo[F], F_1, 0)$ has vanishing differentials.
\end{proof}

\begin{corollary}\label{cor.res}
Let $h : E\sp{*} \to E\sp{*}$ be a linear map, such that $h$ induces an
injective map $E\sp{*}/\ker(h) \to E\sp{*}/\ker(h)$   
Let $f := i \circ h : E\sp{*} \to \calo[E]$, where $i : E\sp{*} \to \calo[E]$ is
the canonical inclusion.  Denote by $F \subset E$ the annihilator of
$h(E\sp{*}) \subset E\sp{*}$. Then the restriction morphisms $\calo[E] \to
\calo[F]$ and $E\sp{*} \to F\sp{*}$ define isomorphisms
\begin{equation*}
    res_{\KK} \, : \, H_q \KK_*(\calo[E], E\sp{*}, f) \, \simeq \, H_q
    \KK_*(\calo[F], F\sp{*}, 0) \, = \, \calo[F] \otimes \wedge^q
    F\sp{*} \, = \, \Omega\sp{p}[F] \;.
\end{equation*}
\end{corollary}

\begin{proof} 
This follows from Proposition \ref{prop.res} for $E_1 = E\sp{*}$ as
follows.  First of all, by assumption, we have that $F_1 : = \ker(h :
E\sp{*} \to E\sp{*})$ has $h(E\sp{*})$ as complement. We shall
therefore choose the projection $E\sp{*} \to F_1$ to vanish on
$h(E\sp{*})$. But $E\sp{*}/h(E\sp{*}) \simeq F\sp{*}$ naturally, by
restricting linear forms on $E$ to $F$, since $F$ was defined as the 
anihilator of $h(E\sp{*})$. 
Thus $F_1 \cong F\sp{*}$ and the projection $E\sp{*} \to
F_1$ becomes simply the restriction $E\sp{*} \to F\sp{*}$.
\end{proof}

Koszul complexes are relevant for Hochschild homology through the
following lemma that is also well known \cite{LodayQuillen}. We denote
as usual by $S_q$ the group of permutations of the set
$\{1,\ldots,q\}$, $q \ge 1$.

\begin{lemma}\label{three} 
Let $i : E\sp{*}
\to \calo[E]$ be the canonical inclusion and  $\delta : E\sp{*}
\to R^{e} := R \otimes R$ be given by the formula $\delta(\xi) = i(\xi) \otimes 1 -
1 \otimes i(\xi)$. Then the Koszul complex $\KK_* (R^{e}, E\sp{*},
\delta)$ is a resolution of $R$ by projective $R^{e}$-modules.  Fix a
basis $e_1, \ldots, e_n$ of $E\sp{*}$ and define $\kappa_0 :
\KK_*(R\sp{e}, E\sp{*}, \delta) \to (\calb_*\rp(R), b\rp)$ by the
formula
\begin{equation*}
  \kappa_0 (a_1 \otimes a_2 \otimes e_{j_1} \wedge \ldots \wedge
  e_{j_p}) \, := \, \sum_{\sigma \in S_p}\mbox{\rm sign}(\sigma)
  \, a_1 \otimes i(e_{j_{\sigma(1)}}) \otimes \ldots \otimes
  i(e_{j_{\sigma(p)}}) \otimes a_2 \,.
\end{equation*}
Then $\kappa_0$ is a chain map, that is, $\kappa_0 \pa = b\rp
\kappa_0$.
\end{lemma}

We provide a proof in our setting, for the benefit of the reader.

\begin{proof} 
Recall that $R$ is commutative, so $R^{\opp} = R$, and hence the
notation $R^{e} = R \otimes R$ is justified. The Koszul complex 
$\KK_* (R^{e}, E\sp{*}, \delta)$ is a resolution of $R$ by projective
$R^{e}$-modules by Corollary \ref{cor.two} which is used for $R_1 = \calo[E]$
and the decomposition $R^e \simeq R_2 \otimes R_1$, where $R_2$ is
the polynomial ring generated by the image of $\delta$, $R_1$ 
is the polynomial ring generated by vectors of the form $i(\xi) \otimes 1
+ 1 \otimes i(\xi)$, $\xi \in E$, and $R_1 \simeq R_2 \simeq R$. 
The fact that $\kappa_0$ is a
chain map is a direct calculation.
\end{proof}

We then obtain

\begin{corollary}\label{cor.three} 
Let $g : E \to E$ be linear and  $f :=i\circ  (g\sp{*}-1) :
E\sp* \to E\sp* \subset \calo[E]$. Let us denote also by $g$ the
endomorphism of $\calo[E]$ induced by $g$. Using the notation of
Lemma \ref{three}, we have that $\kappa_E : (\KK_*(R, E\sp{*}, f),\pa) \to (\calb_*(R), b_g)$
\begin{equation*}
	\kappa_E(a \otimes e_{j_1}\wedge \ldots \wedge e_{j_p}) \, := \,
	\sum_{\sigma \in S_p}\mbox{\rm sign}(\sigma) \, a \otimes
	i(e_{j_{\sigma(1)}}) \otimes i(e_{j_{\sigma(2)}}) \otimes \ldots \otimes
	i(e_{j_{\sigma(p)}})
\end{equation*}
is a quasi-isomorphism (that is, $\kappa_E \pa = b_g \kappa_E$, and it
induces an isomorphism of homology groups).
\end{corollary}

\begin{proof}
Denote $R := \calo[E]$, as above. The groups $\Hd_q(R, g)$ can be
obtained from any projective resolution of  $R$ as  a $R\sp{e}$ module
by tensoring this projective resolution with $R_g$ over $R\sp{e}$, by Proposition
\ref{prop.isom.gt}. We shall use to this end the resolution of Lemma
\ref{three}. Thus $\Hd_q(R, g)$ are the homology groups of the complex
\begin{equation*}
  (\KK_* (R^{e}, E\sp{*}, \delta)\otimes_{R\otimes R} R_g, \pa \otimes 1) \,.
\end{equation*}
By the functoriality of the Koszul complex (Remark
\ref{rem.functoriality}), this tensor product is nothing else but the
Koszul complex $\KK_* (R, E\sp{*}, f)$. Let $\kappa_0$ be as in the statement
of Lemma \ref{three}. We have then that $\kappa_0
\otimes_{R\otimes R} 1_{R_g} = \kappa_E$. Since $\kappa_0$ is a
morphism of projective resolutions by Lemma \ref{three}, it follows
that $\kappa_E$ is a quasi-isomorphism by standard homological algebra.
\end{proof}

We shall need the following lemma, which is a particular
case of Proposition \ref{prop.brylinski}. Some closely related 
results with a Hopf algebra flavor can be found in \cite{Farinati}   
and some applications of it to deformation of algebras are given in \cite{SW}.

\begin{lemma}\label{lemma.brylinski}
We use the notation of Corollary \ref{cor.three} and assume that $g-1:
E/\ker(g-1) \to E/\ker(g-1)$ injective. Let $E\sp{g} := \ker(g-1)$. 
Then the restriction
$\calo[E] \to \calo[E^{g}]$ defines isomorphisms
\begin{equation*}
  res_{\Hd} \, : \, \Hd_q(\calo[E], g) \ \to \ \Hd_q(\calo[E^{g}]) \,,
\end{equation*}
and hence the twisted 
Connes-Hochschild-Kostant-Rosenberg map $\chi_g$ of Equation
\eqref{eq.Hochschild-Kostant-Rosenberg} gives isomorphisms
\begin{equation*}
  \chi_g : \Hd_q(\calo[E], g) \, \cong \, \Omega^q(E^{g}) \;.
\end{equation*}
\end{lemma}

\begin{proof} 
Let $R = \calo[E]$, as before and $f := g\sp{*}-1$. Corollary \ref{cor.three} gives that
the groups $\Hd_q(\calo[E], g)$ are isomorphic to the homology groups
of the Koszul complex $(\KK_* (R, E\sp{*}, f), \pa)$.  Since the
annihilator of $(g\sp{*} -1)(E\sp{*})$ is $\ker(g-1) =: E\sp{g}$, we
have by Corollary \ref{cor.res} that the homology groups of this
Koszul complex are the same as those of the Koszul complex $(\KK_*
(\calo[E\sp{g}], (E\sp{g})\sp{*}, 0)$, with the isomorphism being
given by restriction from $E$ to $E\sp{g}$, a map that we will denote
$res_{\KK} : \KK_* (\calo[E], E\sp{*}, f) \to \KK_* (\calo[E\sp{g}], (E\sp{g})\sp{*}, 0)$.
 
Let us denote by $res_\calb : (\calb(\calo[E]), b_g) \to
(\calb(\calo[E^{g}]), b)$ the natural map given by restriction. Then
we have that the restriction maps $res_\calb$ and the restriction map
$res_{\KK}$ just defined satisfy $res_\calb \circ \kappa_{E} =
\kappa_{E^{g}} \circ res_{\KK}$, with the maps $\kappa$ as in
Corollary \ref{cor.three}.  That same corollary gives that
$\kappa_{E}$ are $\kappa_{E^{g}}$ quasi-isomorphisms.  Since
$res_{\KK}$ is also a quasi-isomophisms, we obtain that $res_\calb$ is
a quasi-isomorphism as well.  Recall that $res_{\Hd}$ is the map
induced at the level of homology by $res_{\calb}$.  That means that
$res_{\Hd}$ is an isomorphism and hence proves the first half of our
statement.

To prove the last part of our statement, we first notice that the
Connes-Hochschild-Kostant-Rosenberg map is an isomorphism $\chi :
\Hd_q(\calo[E^{g}] ) \to \Omega^q(E^{g})$. Hence so is $\chi_g$, as
the composition $\chi_g := \chi \circ res_{\calb}$.
\end{proof}

\subsection{Completion at a maximal ideal}
For the rest of this paper, we shall use completions extensively. Let
us fix a maximal ideal $\mfk m$ of our base ring $\kk$. We denote by
$\widehat{\kk}_{\mfk m}$ the completion of $\kk$ with respect to $\mfk
m$: $\widehat{\kk}_{\mfk m} := \displaystyle{\lim_{\longleftarrow}}\,
\kk/{\mfk m}\sp{n}\kk$.  For the rest of this subsection, all
completions will be with respect to $\mfk m$, and this will be
stressed by including $\mfk m$ as an index.  For any $\kk$-module $M$,
we shall denote by
\begin{equation*}
 \widehat M_{\mfk m} \, := \, \displaystyle{\lim_{\longleftarrow}}\, M/\mfk m^n M
\end{equation*} 
the completion of $M$ with respect to the topology defined by an ideal
$\mfk m \subset \kk$. We notice that the completion $\widehat
M_{\mfk m}$ is naturally a $\widehat{\kk}_{\mfk m}$-module, and hence
we obtain a natural map $comp: M \otimes_{\kk} \widehat{\kk}_{\mfk m}
\to \widehat M_{\mfk m}$.  
In case $M$ is finitely generated as a
$\kk$-module (always the case in what follows), then $comp: M
\otimes_{\kk} \widehat{\kk}_{\mfk m} \to \widehat M_{\mfk m}$ is an
isomorphism \cite{AtiyahMacDonald, BourbakiAlgComm}, which will be
used as an identification from now on. Thus, in order not to
overburden the notation, we shall drop $comp$ from the notation and
simply identify $\widehat {M}_{\mfk m}$ and $M \otimes_{\kk}
\widehat{\kk}_{\mfk m}$.

\begin{remark}\label{remark.e.f}
We know that every element $x \in \kk \smallsetminus \mfk m$ is
invertible in $\widehat{\kk}_{\mfk m}$ by writing a convergent Neumann
series for the inverse. Therefore, for every $\kk$-module $M$, the
canonical map $M \to \widehat M_{\mfk m}$ factors through $M_{\mfk
  m} := (\kk \smallsetminus {\mfk m})\sp{-1}M$, 
  that is, it is the composition of the canonical maps $M \to
M_{\mfk m} \to \widehat M_{\mfk m}$. In particular, if $M_{\mfk m}
=0$, then $\widehat M_{\mfk m}=0$, since the image of $M_{\mfk m}$ in
$\widehat M_{\mfk m}$ is dense. Moreover, we see that completion
with respect to $\mfk m$ and localization at $\mfk m$ commute, so
there is no danger of confusion in the notation $\widehat M_{\mfk m}$:\
that is, $(\widehat M)_{\mfk m} = \widehat {(M_{\mfk m})}$.
\end{remark}

\begin{remark}\label{remark.e.completion}
If $f : M \to N$ is a morphism of $\kk$-modules, we shall denote by
$\widehat f_{\mfk m} : \widehat M_{\mfk m} \to \widehat N_{\mfk m}$
the induced morphism of the corresponding completions with respect to
the topology defined by the powers of $\mfk m$. If $M$ and $N$ are
finitely generated, we shall make no difference between $\widehat
f_{\mfk m} : \widehat M_{\mfk m} \to \widehat N_{\mfk m}$ and $f
\otimes_{\kk} 1 : M \otimes_{\kk} \widehat{\kk}_{\mfk m} \to N
\otimes_{\kk} \widehat{\kk}_{\mfk m}$ Assume $M$ and $N$ are finitely
generated $\kk$-modules.  Then it is a
standard result in commutative algebra that $f$ is an isomorphism if,
and only if, $\widehat f_{\mfk m}$ is an isomorphism for all maximal
ideals $\mfk m$ of $A$ \cite{AtiyahMacDonald, BourbakiAlgComm}.
\end{remark}

Let $A$ be a finite type $\kk$ algebra. We let 
\begin{equation}\label{eq.def.can}
 can \, :\, \Hd_*(A, g) \otimes_\kk \widehat{\kk}_{\mfk m} \, \cong \, 
 \tHd_*(\widehat A_{\mfk m}, g)
\end{equation}
denote the canonical isomorphism of Theorem
\ref{theorem.completion}. It is obtained from 
the fact that $\tHd_*(\widehat A_{\mfk m}, g)$ is a $\widehat{\kk}_{\mfk m}$-module
using also the natural map 
$\Hd_*(A, g) \to \tHd_*(\widehat A_{\mfk m}, g)$ induced by the
inclusion $A \to \widehat A_{\mfk m}$.
We shall use the map $can$ in the following corollary in the following
setting: $E$ will be a finite dimensional, complex vector space, $\kk
:= \calo[E]$, and $\mfk m$ will be the maximal ideal of $\kk :=
\calo[E]$ corresponding to functions vanishing at 0.  Then we shall
consider, as explained, the filtrations defined by $\mfk m$ and denote
by $\widehat \calo[E]_{\mfk m}$ and $\widehat \calo[E\sp{g}]_{\mfk m}$
the completions of $\calo[E]$ and $\calo[E\sp{g}]$ with respect to the
filtrations defined by the powers of $\mfk m$. 

The following corollary 
is an analog of Lemma \ref{lemma.brylinski} for completed algebras.

\begin{corollary}\label{cor.brylinski}
We use the notation and assumptions of Lemma \ref{lemma.brylinski}, in
particular, $g$ is a linear endomorphism of $E$ such that $g-1:
E/\ker(g-1) \to E/\ker(g-1)$ is injective. Let $\mfk m$ be the maximal
ideal of $\kk := \calo[E]$ corresponding to functions vanishing at 0.
Then the restriction $\widehat{\calo}[E]_{\mfk m} \to
\widehat{\calo}[E^{g}]_{\mfk m}$ defines an isomorphism
\begin{equation*}
 \widehat{res}_{\Hd} \, : \, \tHd_q(\widehat{\calo}[E]_{\mfk m}, g) 
  \ \to \ \tHd_q(\widehat{\calo}[E^{g}]_{\mfk m}) \,.
\end{equation*}
Define $\widehat \chi_g$ by $ \chi_g \circ can = \chi_g
\otimes_{\kk} 1$.  Then we have an isomorphism
\begin{equation*}
  \widehat \chi_g : \tHd_q(\widehat \calo[E]_{\mfk m}, g) \ \cong
  \ \widehat{\Omega}^q(E^{g})_{\mfk m} \, := \,
  \displaystyle{\lim_{\longleftarrow}}\ \Omega^q(E^{g})/{\mfk m}\sp{n}
  \Omega^q(E^{g})\;.
\end{equation*}
\end{corollary}

As the notation suggests, $\widehat \chi_g$ is, up to canonical
identifications, nothing but the extension by continuity (or
completion) of the usual Connes-Hochschild-Kostant-Rosenberg map.

\begin{proof} 
For the first part of the proof, let us denote for any $\kk$-algebra
$A$ by $nat$ the natural map $\Hd_*(A, g) \to \tHd_*(\widehat A, g)$.
Hence $can = nat \otimes_{\kk} 1$, see Equation
\eqref{eq.def.can}. Let us consider then the diagram
\begin{equation}\label{diagram.HH}
\begin{CD}
   \Hd_q(\calo[E], g) @>{ res_{\Hd} }>> \Hd_q(\calo[E\sp{g}])
   \\ @V{nat}VV @V{nat}VV \\ \tHd_q(\widehat{\calo[E]}_{\mfk m}, g)
   @>{\widehat {res}_{\Hd}} >> \tHd_q(\widehat{\calo[E\sp{g}]}_{\mfk
     m})
\end{CD}
\end{equation}
whose arrows are given by the natural morphisms of the corresponding
algebras and hence is commutative. That is, we have the relation $nat
\circ res_{\Hd} = \widehat{res}_{\Hd} \circ nat$, which gives then the
relation $ can \circ ({res}_{\Hd} \otimes_{\kk} 1) =
\widehat{res}_{\Hd} \circ can, $ that is, that the diagram
\begin{equation}\label{diagram1}
\begin{CD}
   \Hd_q(\calo[E], g) \otimes_{\kk} \widehat{\kk}_{\mfk m} @>{
     res_{\Hd} \otimes_{\kk} 1 }>> \Hd_q(\calo[E\sp{g}]) \otimes_{\kk}
   \widehat{\kk}_{\mfk m}\\ @V{can}VV
   @V{can}VV\\ \tHd_q(\widehat{\calo[E]}_{\mfk m}, g) @>{\widehat
     {res}_{\Hd}} >> \tHd_q(\widehat{\calo[E\sp{g}]}_{\mfk m})
\end{CD}
\end{equation}
is commutative. The map ${res}_{\Hd} : \Hd_q(\calo[E], g) \to
\Hd_q(\calo[E^{g}])$ is an isomorphism by Lemma \ref{lemma.brylinski}.
Therefore ${res}_{\Hd} \otimes_{\kk} 1 : \Hd_q(\calo[E], g)
\otimes_\kk \widehat{\kk}_{\mfk m} \to \Hd_q(\calo[E^{g}]) \otimes_\kk
\widehat{\kk}_{\mfk m}$ is an isomorphism as well.  Since the vertical
arrows (that is, the maps $can$)
are isomorphisms, we obtain that $\widehat{res}_{\Hd}$ is an
isomorphism as well.

The fact that $\widehat \chi_g$ is an isomorphism follows from the
commutative diagram
\begin{equation}\label{diagram.chi}
\begin{CD} 
   \Hd_q(\calo[E], g) \otimes_{\kk} \widehat{\kk}_{\mfk m} @>{\chi_g
     \otimes_{\kk} 1 }>> \Omega^q(E^{g}) \otimes_{\kk} \widehat
   \kk_{\mfk m}\\ @V{can}VV @| \\ \tHd_q(\widehat{\calo[E]}_{\mfk m},
   g) @>{\widehat \chi_g} >> \widehat{\Omega}^q(E^{g})_{\mfk m}
\end{CD}
\end{equation}
(that is, $\widehat \chi_g \circ can = \chi_g \otimes_{\kk} 1$) and
the fact that $can$ and $\chi_g$ are isomorphisms.
\end{proof}

In plain terms, one has that $\widehat{\chi}_g$ is an isomorphism since
it is the completion of an isomorphism.

We now come back to the case of a general smooth, complex algebraic
variety $X$. We obtain the following result (due to Brylinski in the
case of the algebra of smooth functions \cite{Brylinski}).  See also
\cite{DolgushevAND}

\begin{proposition}\label{prop.brylinski} \
Let $X$ be a smooth, complex, affine algebraic variety and $g$ an
endomorphism of $\calo[X]$ such that $X\sp{g}$ is also a smooth
algebraic variety and such that, for any fixed point $x \in X\sp{g}$,
$T_xX\sp{g}$ is the kernel of $g_* - 1$ acting on $T_xX$ and $g_* - 1$
induces an injective endomorphism of $T_xX/T_xX\sp{g}$.  Then the
twisted Connes-Hochschild-Kostant-Rosenberg map $\chi_g$ induces
isomorphisms
\begin{equation*}
  \chi_g \, : \, \Hd_q(\calo[X], g) \ \cong \ \Omega^q(X^{g})\;.
\end{equation*}
\end{proposition}

\begin{proof} Recall that, for a finitely generated $\kk$-module
$M$, we identify $\widehat M_{\mfk m}$ with $M \otimes_{\kk} \widehat
  \kk_{\mfk m}$. However, in this proof, it will be
  convenient to work with the tensor product $M \otimes_{\kk} \widehat
  \kk_{\mfk m}$ rather than with the completion $\widehat M_{\mfk m}$,
  for notational simplicity.  Denote $\kk = \calo[X]$ and let $\mfk m$
  be an arbitrary maximal ideal of $\kk$. For the clarity of
    the presentation, we shall write in this proof $\chi_g\sp{X}$
    instead of simply $\chi_g$.  As explained in Remark
  \ref{remark.e.completion}, it is enough to check that the map
\begin{equation}\label{eq.?iso}
  \chi_g\sp{X} \otimes_{\kk} 1\, : \, \Hd_q(\calo[X], g) 
  \otimes_{\kk} \widehat{\kk}_{\mfk m} \ \to \ 
  \Omega^q(X^{g}) \otimes_{\kk} \widehat{\kk}_{\mfk m}\;.
\end{equation}
is an isomorphism (since the maximal ideal $\mfk m$ of $A$ was chosen
arbitrarily). We shall now check that this property is satisfied.

Indeed, if $\mfk m$ is not fixed by $g$ (recall that the maximal
ideals ${\mfk m}$ of $\calo[X]$ and the points of $X$ are in one-to-one
correspondence), then $\chi\sp{X}_g \otimes_{\kk} 1$ is an isomorphism since
both groups in Equation \eqref{eq.?iso} are zero, by Corollary
\ref{cor.zero} and Remark \ref{remark.e.f}.

Let us therefore assume that $\mfk m$ is fixed by $g$
and denote by $E:= ({\mfk m}/{\mfk m}\sp{2})\sp{*}$
the tangent space to $X$ at $\mfk m$. The assumption
that $X$ is smooth at $\mfk m$ means, by definition
\cite{HartshorneBook}, that we have a natural isomorphism $tan:
\widehat \calo[X]_{\mfk m} \simeq \widehat \calo[E]_{\mfk m}$.
Similarly, we have a natural isomorphism
\begin{equation*}
 \Omega^q(X^{g}) \otimes_{\kk} \widehat{\kk}_{\mfk m} \, \simeq \,
 \widehat{\Omega}^q(X^{g})_{\mfk m} \, \simeq \,
 \widehat{\Omega}^q(E^{g})_{\mfk m} \, \simeq \, \Omega^q(E^{g})
 \otimes_{\kk} \widehat{\kk}_{\mfk m} \,,
\end{equation*}
which we shall also denote by $tan$.

Recall that the canonical map
$can : \Hd_q(\calo[X], g) \otimes_{\kk} \widehat{\kk}_{\mfk m} \to
\tHd_q(\widehat{\calo[X]}_{\mfk m}, g)$ is an isomorphism (by Theorem
\ref{theorem.completion}) and consider then the commutative diagram
%
%
\begin{equation}\label{diagram}
\begin{CD}
  \Hd_q(\calo[X], g) \otimes_{\kk} \widehat{\kk}_{\mfk m} 
  @>{\chi_g\sp{X} \otimes_{\kk} 1 }>> 
  \Omega^q(X^{g}) \otimes_{\kk} \widehat{\kk}_{\mfk m} \\
   @V{can}VV @| \\
  \tHd_q(\widehat{\calo[X]}_{\mfk m}, g) @>{\widehat \chi_g\sp{X}}>>
  \widehat{\Omega}^q(X^{g})_{\mfk m}\\
  @V{tan}VV @V{tan}VV \\
  \tHd_q(\widehat{\calo[E]}_{\mfk m}, g) @>{\widehat \chi_g\sp{E}} >>
  \widehat{\Omega}^q(E^{g})_{\mfk m}
\end{CD}
\end{equation}
in which the vertical arrows are isomorphisms as explained. 

In order to complete the proof, it is enough to check that 
the bottom horizontal arrow is also an isomorphism. The problem is
that $g$ is not a linear map. Let us denote by $Dg : T_{\mfk m}E
\to T_{\mfk m}E$ the differential of the map $g$ at $\mfk m$
and consider the chain maps
\begin{equation}\label{eq.Omega}
 \begin{gathered}
  \chi_g\sp{E} : (\calb_*(\calo[E]), b_g)  \, \to \,  
  (\Omega^*(E^{g}), 0) \ \ \mbox{ and}\\
  \chi_{Dg}\sp{E} : (\calb_*(\calo[E]), b_{Dg})  \, \to \, 
  (\Omega^*(E^{g}), 0)
 \end{gathered}
\end{equation}
of filtered complexes (with the filtration defined by the powers
of $\mfk m$). The second map is a quasi-isomorphism by 
Corollary \ref{cor.brylinski}. Recall that each of the three filtered
complexes in Equation \eqref{eq.Omega} gives rise to a spectral sequence
with $E\sp{1}$ term given as the homology of the subquotient 
complexes. We have that the corresponding $E\sp{1}$-terms and maps
between these $E\sp{1}$-terms depend only on $Tg$. 
Since $\chi_{Dg}\sp{E}$ induces a quasi-isomorphism, the result
follows from Lemma \ref{lemma.qi.complete}.
\end{proof}

We note that the assumptions of our proposition are satisfied for $g$
a finite order automorphism.

We have the following analog of a result of \cite{Hubl1}, where from
now on the completions are with respect to an ideal $I$, and hence the
subscript ${\mfk m}$ will be dropped from the notation.

\begin{theorem}\label{thm.deRham} \
Let $A = \calo[X]$ with $X$ a complex, smooth, affine algebraic variety,
$V \subset X$ a subvariety with ideal $I$ and $g$ 
finite order automorphism of
$\calo[X]$ that leaves $V$ invariant. Then  the twisted
Connes-Hochschild-Kostant-Rosenberg map $\chi_g$ induces an
isomorphism
\begin{equation*}
  \chi_g \, : \, \tHd_q(\widehat{\calo[X]}, g) \ \cong
  \ \widehat{\Omega}^q( V^{g} ) \, := \,
  \displaystyle{\lim_{\longleftarrow}}\ \Omega^q(X^{g})/I^n
  \Omega^q(X^{g})\;,
\end{equation*}
where the completions are taken with respect to the powers of $I$. 
\end{theorem}

\begin{proof} 
This follows by taking the $I$-adic completion of the map $\chi_g$ of
Proposition \ref{prop.brylinski} and then using Theorem
\ref{theorem.completion}.
\end{proof}

\subsection{Crossed products}

We now return to cross products. The cyclic homology of crossed
products was extensively studied due to its connections with
non-commutative geometry \cite{ConnesBook}. See \cite{C7, C6, C1, C4,
  hochschildPadic, C2, C3, C5} for a few recent results. The homology
of cross-product algebras was studied recently by Dave in relation to residues
\cite{Shantanu} and by Manin and Marcolli in relation to arithmetic
geometry \cite{Manin1}. In this section, we compute the
homology of $\calo[V]\rtimes \Gamma$, for $V$ an algebraic variety
(not necessarily smooth) and $\Gamma$ a finite group. The idea is to 
reduce to the case when $V$ is the affine space $\CC\sp{n}$ acted upon
linearly by $\Gamma$. In this case the Hochschild and cyclic homology
groups of $\calo[V]\rtimes \Gamma$ were computed in \cite{Farinati, SW},
but we need a slight enhacement of those results, which are provided
by the results of the previous subsection.

Let us therefore fix an affine, complex algebraic variety $V$ and let
$\Gamma$ be a finite group acting by automorphisms on $\calo[V]$. Let
us choose a $\Gamma$-equivariant embedding $V \to X$, where $X$ is a
smooth, affine algebraic variety on which $\Gamma$ acts by regular
maps. Propositions \ref{prop.decomp} and \ref{prop.brylinski} then
give the following result (see also \cite{DolgushevAND, Farinati, SW}).

\begin{proposition}\label{proposition.cross}\
Let $\Gamma$ be a finite group acting on a smooth, complex, affine algebraic
variety $X$ and $\{\gamma_1 , \dots , \gamma_\ell\}$ be a list of
representatives of its conjugacy classes. We denote by $C_j$ the
centralizer of $\gamma_j$ and by $X_j$ be the set of fixed points of
$\gamma_j$. Then
\begin{equation*}
 \Hd_q(\calo[X] \rtimes \Gamma) \ \cong \ \bigoplus_{j=1}^\ell \,
 \Omega^q(X_j)^{C_{j}} ,
\end{equation*}
%
 \begin{equation*}
   \Hc_q(\calo[X] \rtimes \Gamma) \ \cong \ \bigoplus_{j=1}^\ell \,
   \left( \Omega^q(X_j)^{C_{j}}/ d\Omega^{q-1}(X_j)^{C_{j}}\, \oplus
   \, \bigoplus_{k \ge 1} H^{q-2k}( X_j)^{C_{j}} \right) \,,
\end{equation*}
	and 
\begin{equation*}
  \Hp_q(\calo[X] \rtimes \Gamma) \ \cong \ \bigoplus_{j=1}^\ell \,
  \bigoplus_{k \in \ZZ} \, H^{q-2k}(X_j)^{C_{j}} \, .
\end{equation*}
\end{proposition}

\begin{proof}
The result for Hochschild homology follows from Propositions
\ref{prop.decomp} and \ref{prop.brylinski}. Let $g = \gamma_j$, for
some $j$. The result for Hochschild homology, together with the
equation $\chi_g B_g = d \chi_g$, with $d$ the de Rham differential,
then gives the other isomorphisms using a standard argument based on
mixed complexes. Mixed complexes were introduced in
\cite{JonesKassel89}. They are also reviewed in
\cite{KazhdanNistorSchneider, SeibtBook}.

Let us recall now this standard argument based on mixed complexes.
For any $\gamma \in \Gamma$ with centralizer $C_\gamma$, the twisted
Connes-Hochschild-Kostant-Rosenberg map $\chi_\gamma$ defines a map of
mixed complexes
\begin{equation*}
  \chi_\gamma \, : \, \big (\calb_*(\calo[X])^{C_\gamma}, b_\gamma,
  B_\gamma \big ) \to (\Omega^*(X^\gamma)^{C_\gamma}, 0, d) \,,
\end{equation*}
which is an isomorphism in Hochschild homology by Proposition
\ref{prop.brylinski}.  It follows that $\chi_\gamma$ induces an
isomorphism of the corresponding cyclic and periodic cyclic homology
groups. Since the cohomology groups of the de Rham complex
$\Omega\sp{*}(X)$ identify with $H^*(X; \CC)$, the singular cohomology
groups of $X$ with complex coefficients, \cite{EmmanouilCR, Emmanouil,
  HartshorneDR, Hubl1} and these groups vanish for $*$ large.
\end{proof}

Similarly, we obtain

\begin{theorem}\label{theorem.cross}\
Let $\Gamma$ be a finite group acting on a smooth, complex, affine
algebraic variety $X$ and $V \subset X$ be an invariant subvariety.
Let $\{\gamma_1 , \dots , \gamma_\ell\}$ be a list of representatives
of conjugacy classes of $\Gamma$, let $C_j$ be the centralizer of
$\gamma_j$, and let $X_j$ and $V_j$ be the set of fixed points of
$\gamma_j$.  We complete with respect to the powers of the ideal $I
\subset \calo[X]$ defining $V$ in $X$.  Then we have
\begin{equation*}
\begin{gathered}
 \tHd_q(\widehat{\calo[X]} \rtimes \Gamma) \ \cong
 \ \bigoplus_{j=1}^\ell \, \widehat{\Omega}^q(X_j)^{C_{j}} \\
 \tHc_q(\widehat{\calo[X]} \rtimes \Gamma) \ \cong
 \ \bigoplus_{j=1}^\ell\, \left( \, \widehat{\Omega}^q(X_j)^{C_{j}}/
 d\widehat{\Omega}^{q-1}(X_j)^{C_{j}} \ \oplus\ \bigoplus_{k \ge 1} \,
 H^{q-2k}( V^{j}; \CC)^{C_{j}} \, \right)\\
  \tHp_q(\widehat{\calo[X]} \rtimes \Gamma) \ \cong
  \ \bigoplus_{j=1}^\ell \, \bigoplus_{k \in \ZZ} \, H^{q-2k}(V^{j};
  \CC)^{C_{j}} \,.
\end{gathered}
\end{equation*}
\end{theorem}

\begin{proof}
The result for Hochschild homology follows from Proposition
\ref{prop.decomp2} and Theorem \ref{thm.deRham}.  The rest for is very
similar, except that one has to use also the fact that the de Rham
complex $\widehat{\Omega}^{*}(V) := \displaystyle{\lim_{\leftarrow}} \,
\Omega(X)/ I\sp{k} \Omega(X)$ endowed with the de Rham differential
has cohomology $H\sp{*}_{\rm inf}(V) \simeq H^*(V; \CC)$ 
\cite{EmmanouilCR, Emmanouil, FeiginTsygan, HartshorneDR, Hubl1}.
\end{proof}

This result was obtained in the case of smooth functions in 
\cite{BrylinskiNistor}.
We now come to our main result. Note that in the following theorem we
do not assume the variety $V$ to be smooth and that we recover the
orbifold homology groups. The case when $\Gamma$ is trivial is due to
Feigin and Tsygan \cite{FeiginTsygan} (see also Emmanouil\rp s papers
\cite{EmmanouilCR, Emmanouil} and the references therein) for
the case when $V$ is smooth, see also the paper by Dolgushev and Etinghof
\cite{DolgushevAND}. See also  \cite{BaiLiWang}.

\begin{theorem}\label{theorem.cross2}\
Let $\Gamma$ be a finite group acting on a complex, affine algebraic
variety $V$. We do not assume $V$ to be smooth. 
Let $\{\gamma_1 , \dots , \gamma_\ell\}$ be a list of
representatives of conjugacy classes of $\Gamma$, let $C_j$ be the
centralizer of $\gamma_j$, and let $V_j \subset V$ be the set of fixed
points of $\gamma_j$.  Then
\begin{equation*}
  \Hp_q(\calo[V] \rtimes \Gamma) \ \cong \ \bigoplus_{j=1}^\ell \,
  \bigoplus_{k \in \ZZ} \, H^{q-2k}(V_{j}; \CC)^{C_{j}} \, .
\end{equation*}
\end{theorem}

\begin{proof}
Let us choose a $\Gamma$-equivariant embedding $V \subset E$ into an
affine space. This can be done by first choosing a finite system of
generators $I := \{a_i\}$ of $\calo[V]$ and then replacing it with
$I_\Gamma := \{\gamma a_i\}$, for all $\gamma \in \Gamma$. We let $E$
to be the vector space with basis $(\gamma, a_i)$. Let $I \subset \calo[E]$
be the ideal defining $V$ in $E$.

We let $\kk = \calo[E]^\Gamma = \calo[E/\Gamma]$, then
$\kk$ is a finitely generated complex algebra by Hilbert\rp s finiteness
theorem \cite{Potier}.
Moreover, $\calo[E]$ is a
finitely generated $\kk$-module, and hence $\calo[E] \rtimes \Gamma$
is also a finite type $\kk$-algebra. Let $I_0 := \kk \cap I$.
We also notice that there exists $j$ such that $I^j
\subset I_0 \calo[E] \subset I$. Indeed, this is due to the fact that if
a character $\phi : \calo[E] \to \CC$ vanishes on $I_0$, then it vanishes on $I$
as well, so $I$ and $I_0 \calo[E]$ have the same nilradical. To prove this, let
us consider $\phi : \calo[E] \to \CC$ that vanishes on $I_0$ and let $a \in I$.
Since $a \in I$ and $I$ is $\Gamma$-invariant, 
the polynomial $P(X) := \prod_{\gamma \in \Gamma} (X - \gamma(a))$ has
coefficients in $I$ that are invariant with respect to $\gamma$, so they are
in $I_0$. This gives $\phi(P(X)) = X\sp{n} 
= \prod_{\gamma \in \Gamma} (X - \phi(\gamma(a)))$, with $n$ the number of elements
in $\Gamma$. Therefore $\phi(\gamma(a)) = 0$ for all $\gamma$. In particular,
$\phi(a) = 0$.

The fact that there exists $j$ such that $I^j
\subset I_0 \calo[E] \subset I$ gives that completing with respect to $I
\rtimes \Gamma$ or with respect to $I_0$ has the same effect. Then
$\calo[E] \rtimes \Gamma/I \rtimes \Gamma \simeq \calo[V] \rtimes
\Gamma$.  Therefore $\tHp_q (\calo[E] \rtimes \Gamma) \simeq \Hp_q
(\calo[V] \rtimes \Gamma)$, by Seibt's theorem (Theorem
\ref{Seibt}). The result then follows from Theorem
\ref{theorem.cross}.
\end{proof}

This theorem extends the calculations for the cross-products of the form
$\calc^{\infty}(X) \rtimes \Gamma$, see \cite{BlockGetzler, 
  BrylinskiNistor, Brylinski,  NestCyclicR, FeiginTsygan, 
  NestCyclicZ, nistorInvent90, PongeIII}.  A similar result
exists for orbifolds in the $\calc^\infty$-setting (often the
resulting groups are called ``orbifold cohomology'' groups). It would
be interesting to extend our result to ``algebraic orbifolds''.
Theorem \ref{cor.cross2} now follows right away. Since for a finite
type algebra $A$ and $\Gamma$ finite, $A \rtimes \Gamma$ is again a 
finite type algebra, it would be interesting to compare the result
of Theorem \ref{theorem.cross2} with the spectral sequence of 
\cite{KazhdanNistorSchneider},
which are based on the excision exact sequence in cyclic homology
\cite{CuntzQuillenJAMS, CuntzQuillenInvent, MeyerExcision}. Cyclic cohomology for 
various algebras (among which cross-products play a central
role) has played a role in index theory \cite{
BaumConnes1, 
ConnesBook,
ConnesSkandalis, 
ConnesMoscoviciNovikov, LeschMoscoviciPflaum, PerrotRodsphon, PongeI, 
PongeII}.

We can now complete the proof of one of our main theorems, Theorem \ref{cor.cross2}.

\begin{proof} (of Theorem \ref{cor.cross2}).
Let $I$ be the nilradical of $A$.  Then $A/I \simeq \calo[V]$ and $I$
is nilpotent.  The result then follows from Goodwillie\rp s result,
Theorem \ref{Goodwillie} and Theorem \ref{theorem.cross2}.
\end{proof}

\begin{remark}
The above result is no longer true if we replace $A = \calo[V]$ with
an Azumaya algebra. Indeed, let $A = M_2(\CC)$ and let $\Gamma :=
(\ZZ/2\ZZ)\sp{2}$ act on $A$ by the inner automorphisms induced by the
matrices
\begin{equation*}
 u_1 \, := \, \left [ 
 \begin{array}{cc} 
    \ 0 & 1 \ \\
    \ 1 & 0 \
 \end{array}
 \right ] \quad \mbox{and} \quad
 u_2 \, := \, \left [ 
 \begin{array}{cc} 
    {\ 1 \ }  & {\ 0 \ } \\
    {\ 0 \  } &  -1
 \end{array}
 \right ]  \;.
\end{equation*}
Then $A \rtimes \Gamma \simeq M_4(\CC)$ by the morphism defined by
\begin{equation}
 M_2(\CC) \ni a \to a \otimes 1 \in M_2(\CC) \otimes M_2(\CC) \simeq M_4(\CC)
\end{equation}
and by $u_i \to u_i \otimes u_i$. Therefore the periodic cyclic
homology of $A \rtimes \Gamma$ is concentrated at the identity. 
\end{remark}

This remark is related to a result of P. Green. See the beginning of the 
introduction to \cite{ErpWilliams} for a statement of this result in the form 
we need it (but with $\ZZ_2$ replaced by the unit circle group $S\sp{1} = \TT$).

\section{Affine Weyl groups}
\label{sec.weyl}

We now use the general results developed in previous sections to
determine the periodic cyclic cohomology of the group algebras of
(extended) affine Weyl groups. These results continue the results in
\cite{BaumNistor}, where the corresponding Hecke algebras were
studied. In fact, cyclic and Hochschild homologies behave remarkably
well for the algebras associated to reductive $p$-adic groups
\cite{RG5, BaumNistor, Crisp, KazhdanNistorSchneider, RG1, RG2, RG3, RG4}. 
Some connections with the Langlands program were pointed
out in \cite{langlands}.

An (extended) affine Weyl group is the crossed product $W = X \rtimes
W_0$, where $W_0$ is a finite Weyl groups and $X$ is a sublattice of
the lattice of weights of a complex algebraic group with Weyl group
$W_0$. Then $\CC[W]$, the group algebra of $W$,
satisfies $\CC[W] \cong \calo(X^*) \rtimes W_0$, that is, it is
isomorphic to the crossed product algebra of the ring of regular
functions on $X^* := \Hom(X, \CC^*)$, the dual torus of $X$, by the
action of $W_0$. The cyclic homology groups of 
$\CC[W] \cong \calo(X^*) \rtimes W_0$
thus can be computed in two, dual ways, either using the results of
\cite{BurgheleaGr, ConnesIHES, Karoubi87} on the periodic cyclic
homology of group algebras or using our determination in Theorem
\ref{theorem.cross2}.

As a concrete example, let us compute
the periodic cyclic cohomology of the group algebra $\CC[W]$ of the
extended Weyl group $W := \ZZ^n \rtimes S_n$, the symmetric Weyl group
$S_n$ acting by permutation on the components of $\ZZ^n$.  Denote by
\begin{equation*}
	\Pi(n) \, := \, \{ (\lambda_1,\lambda_2,\ldots,
        \lambda_r),\ \lambda_1 \ge \lambda_2 \ge \ldots \ge \lambda_r
        > 0, \ \sum \lambda_j = n,\ \lambda_i \in \ZZ \, \}
\end{equation*}
the set of partitions of $n$. The set $\Pi(n)$ is in bijective
correspondence with the set of conjugacy classes of $S_n$.  If
$S_\lambda \simeq S_{\lambda_1} \times S_{\lambda_2} \times \ldots
\times S_{\lambda_r}$ denotes the Young (or parabolic) subgroup of
$S_n$ consisting of permutations leaving the first $\sum_{j = 1}^k
\lambda_j$ elements invariant, for all $k$, then
to the partition $\lambda$ there corresponds a conjugacy class in
$S_n$ represented by a permutation $\sigma_\lambda \in S_\lambda$,
which in each factor is the cyclic permutation $(l,l+1, \ldots, m)$ of
maximum length, for suitable integers $l < m$ (more precisely $l =
\lambda_1 + \lambda_2 + \ldots + \lambda_k +1$ and $m = \lambda_1 +
\lambda_2 + \ldots + \lambda_{k +1}$, for a suitable $k$). Let
$\sigma_\lambda$ be that element.  Then the set of fixed points of
$\sigma_\lambda$ acting on $(\CC^*)^n$ is a tours naturally identified
with $(\CC^*)^r$.

The centralizer $(S_n)_{\sigma_\lambda}$ is isomorphic to a
semi-direct product of a permutation group, denoted $Q_\lambda$,
which permutes the equal length cycles of $\sigma_\lambda$, and the
commutative group generated by the cycles of $\sigma_\lambda$. The
action of the centralizer $(S_n)_{\sigma_\lambda}$ on the set
$(X\sp{*})\sp{\sigma_\lambda} \simeq (\CC^*)^r$ of fixed points of
$\sigma_\lambda$ descends to an action of $Q_\lambda$ on that set.
The group $Q_\lambda$ permutes the equal-length cycles of $\sigma_\lambda$
and acts accordingly by permutations on $(\CC^*)^r$. Let
$t(\lambda)$ be the number of permutation subgroups appearing as
factors of $Q_\lambda$. Thus, $t(\lambda)$ is the set of distinct
values in the sequence $\{\lambda_1, \lambda_2, \ldots, \lambda_r\}$.
Then $(\CC^*)^r/Q_\lambda$ is the product of $(\CC^*)^{t(\lambda)}$
and a euclidean space, so $\cohom^*((\CC^*)^r/Q_\lambda) \cong
\cohom^*((\CC^*)^{t(\lambda)})\cong
\cohom^*(\TT^{t(\lambda)})$, with $\TT \simeq S^{1}$ the unit circle..

\begin{theorem}\label{theorem.GLn}\
Let $A = \CC[W]$ be the group algebra of the extended affine Weyl
group $W=\ZZ^n \rtimes S_n$. Then
\begin{equation*}
 \Hp_q(A) \cong \bigoplus_{\lambda \in \Pi(n)}\,
 \bigoplus_{k=0}\sp{[t(\lambda)/2]} \ \cohom^{2k + q}
 \big (\TT^{t(\lambda)} \big )\;, \quad q = 0, 1\; .
\end{equation*}
\end{theorem}

See also \cite{brodzkiPlymenBull02, brodzkiPlymenDoc02} for some
related results. Theorem 3.1 and Lemma 3.2 in \cite{brodzkiPlymenDoc02} extends 
some of these results to Hecke algebras.

\bibliographystyle{amsplain}
\bibliography{cyclic,nistor}

\end{document}